\pdfoutput=1
\documentclass[12pt]{article}
\usepackage[utf8]{inputenc}
\usepackage{graphicx}
\usepackage{amsmath}
\usepackage{blkarray}
\usepackage{fullpage}
\usepackage{mathtools}
\usepackage{commath}
\usepackage{csquotes}
\usepackage{amsfonts}
\usepackage{amssymb}
\usepackage{amsthm}
\usepackage{tikz}
\usepackage{tikz-network}
\usetikzlibrary{hobby}
\usepackage{multicol}
\usepackage{mathrsfs}
\usepackage{multicol,cite}

\usepackage{caption}
\usepackage{subcaption}

\usepackage{multicol}

\newtheorem{theorem}{Theorem}[section]
\newtheorem{lemma}{Lemma}[section]

\newtheorem{example}{Example}[section]
\newtheorem{definition}{Definition}[section]

\newtheorem{remark}{Remark}[section]

\DeclareMathOperator{\diam}{diam}
\DeclareMathOperator{\rad}{rad}
\DeclareMathOperator{\rank}{rank}

\DeclareMathOperator{\In}{In}
\title{Inertia and spectral 
	symmetry of the eccentricity matrices of a class of bi-block graphs }
\author{T. Divyadevi\thanks{tdivyadevi@gmail.com} \quad  and \quad I. Jeyaraman\thanks{jeyaraman@nitt.edu}\\\\
	Department of Mathematics\\
	National Institute of 
	Technology		Tiruchirappalli-620 015,  India 
}

\date{}
\begin{document}
\maketitle

\begin{abstract}
\noindent	The eccentricity matrix of a simple connected graph 
$G$ is obtained from the distance matrix of $G$ 
by retaining the largest non-zero distance in each row and column, and 
the remaining entries are defined to be zero. A bi-block graph is a simple  
connected 
graph whose blocks are all complete bipartite graphs with  possibly 
different orders. In this paper, we study the eccentricity matrices of a 
subclass $\mathscr{B}$  (which includes trees) of bi-block graphs. We first find the 
inertia of the eccentricity matrices of graphs in $\mathscr{B}$, and thereby  
we  
characterize graphs in $\mathscr{B}$ with odd diameters. Precisely, if $G\in \mathscr{B}$ with diameter of $G$ greater than three, then we show 
that 
the eigenvalues of the eccentricity matrix of $G $ are symmetric 
with respect to the origin if and only if  the diameter of $G$ is 
odd. Further, we 
prove that the eccentricity matrices of  graphs in $\mathscr{B}$ 
are irreducible.
\end{abstract}



\section{Introduction}\label{sec1}

Let $G$ be a simple connected graph on $n$ vertices with the vertex set 
$V_G=\{v_1, v_2,\ldots , v_n\}$. 
In literature, many properties and structural characteristics of $G$ have been 
studied by associating several matrices corresponding to $G$, see 
\cite{Bapa2010}.
Let us recall some of them which are pertinent to our discussion 
here. The \textit{adjacency matrix} of $G$, denoted by
$A(G)$, is an $n \times n$ matrix whose $(i,j)$-th entry is one if
$v_i$ and $v_j$ are adjacent and zero elsewhere.  Let $d(v_i,v_j)$ denote the 
length of a shortest path between the vertices $v_i$ and $v_j$.
The \textit{distance 
	matrix} $D(G):=(d_{ij})$ of $G$, is an $n \times n$ matrix 
such that  
$d_{ij}=d(v_i,v_j)$ for all $i$ and $j$. 
Clearly, the 	adjacency matrix can be obtained 
from the 
distance matrix by retaining  the smallest non-zero distance (which is equal 	
to one) in each row and column and by setting the 	rest of the  entries equal
to
zero. Inspired by this, 
Randi\'{c} \cite{Rand2013} associated a new  matrix corresponding to $G$, 
namely
$D_{\text{MAX}}$,  which is derived from the distance 
matrix $D(G)$ by keeping the largest non-zero distance in each row and each 
column of $D(G)$
and defining the remaining  entries equal to zero. Later, Wang et al. \cite{WLBR2018} obtained an equivalent definition of 
$D_{\text{MAX}}$ using the  notion of  
eccentricity of a vertex and called it  the eccentricity matrix. Recall that
the \textit{eccentricity} of a vertex $v_i$ is defined by
$e(v_i)$= max $\{d(v_i,v_j): v_j\in V_G\}$.
The \textit{eccentricity matrix} $E(G)$ of $G$ is an $n \times n$ 
matrix whose $(i,j)$-th entry is given by
\begin{eqnarray*}\label{Eccentricty matrix defintion}
	E (G)_{ij}= \begin{cases}  d(v_i,v_j)   & \text{if } 
		d(v_i,v_j)=\text{min}\{e(v_i),e(v_j)\},\\
		0 & \text{otherwise}.
	\end{cases}
\end{eqnarray*}
This matrix has been well studied in the literature (see 
\cite{LiWB2022,MaKa2022,MGKA2020,MGRA2021,PaSP2021,Rand2013,WLBR2018,WLWL2020,WLLB2020})
and has applications in Chemistry, see 	\cite{Rand2013,WLBR2018, WLWL2020}.

It is significant to note that
the 
adjacency  and  distance matrices of connected graphs are always 
irreducible, 
but the eccentricity matrix fails to satisfy this property for a general 
connected 
graph.  For example, the eccentricity matrix of a complete bipartite graph $K_{l,m}$  is reducible for all $l,m\geq 2$, see  \cite{WLBR2018}.  However, for some classes of graphs, the associated
eccentricity matrices are irreducible. It is proved in  
\cite{WLBR2018} that the 
eccentricity matrix  of a tree with at least two vertices 
is 
irreducible, and an alternative proof of this result is given in 
\cite{MGKA2020}. This result has been extended to some larger classes of graphs 
in  \cite{LiWB2022} and \cite{WLLB2020}. It is shown in \cite{PaSP2021} that the eccentricity 
matrices of 
the coalescence of complete graphs are irreducible.
Characterizing the graphs whose  eccentricity matrices 
are irreducible, posed by Wang et al. \cite{WLBR2018},  remains an open problem. 
In this paper, we provide  a class of bi-block graphs 
$\mathscr{B}$, defined below,  whose
eccentricity matrices are irreducible.

To introduce the graph class $\mathscr{B}$, let us recall the following
definitions. 
A \textit{block} of a
graph $G$ is a maximal connected subgraph of $G$ which has no cut-vertex. 
A 
connected graph  $G$ is called a \textit{bi-block graph} 
(\textit{block 
	graph}) if all its blocks are  complete bipartite
graphs (respectively, complete graphs) of	possibly varying orders.   Note that the complete bipartite graphs are  bi-block graphs with exactly one 	block.  Since the spectral properties of the 
eccentricity matrices of complete  bipartite graphs are already explored in \cite{MGKA2020}, it 
is sufficient to consider  the   bi-block graphs  with at least two blocks. The main focus of this article is to analyze the spectral symmetry of the eccentricity matrices of   bi-block graphs. In Section \ref{inertia and spectral symmetry}, it is shown that the spectral symmetry result is not true for a  bi-block graph having more than two cut-vertices in a block. In view of these, we define the class $\mathscr{B}$ which is  the collection of all bi-block graphs  with at least two blocks and at most two cut-vertices in each block. Clearly, $\mathscr{B}$ contains all trees with at least three vertices.

In what follows, we present a brief survey of the literature 
on bi-block graphs. In \cite{DaMo2021}, the authors studied the spectral 
radius 	of the  	adjacency matrix of a  class of bi-block graphs  with a 	given 
independence number. The permanent, determinant, and the rank of the 
adjacency matrices of  bi-block graphs were computed in \cite{Sing2020}.
The inverse formula given for the distance matrix of a 
tree was extended to 
a 	subclass of bi-block graphs in \cite{HoSu2016}. 	Inspired by 
these, in this 
article, we 
study the eccentricity matrices of the subclass $\mathscr{B}$ of bi-block graphs.

A primary motivation for studying 
this paper	comes from the  following spectral symmetry results.
In related to this, we first recall a well-known characterization for  bipartite graphs. A graph $G$ is a bipartite graph if and only if 
the 	eigenvalues of the adjacency matrix $A(G)$ are 
symmetric with respect to the origin (i.e., if $\mu$ is an eigenvalue of 
$A(G)$ with multiplicity $l$ then -$\mu$ is also an eigenvalue of $A(G)$ with multiplicity $l$).	In \cite{MaKa2022}, the authors 
studied a similar equivalence for the eccentricity matrices of trees and 
gave a characterization for trees of odd diameters.
Specifically, 
the eigenvalues of  $E(T)$ of a tree $T$  are symmetric with respect to the 
origin if and only if $\diam(T)$ is odd.  An analogous result has been 
proved for a  
subclass of 
block graphs (namely, clique 
trees) in  \cite{LiWB2022}.
The 
problem of characterizing the graphs whose eigenvalues of the eccentricity matrices are symmetric with respect to the origin, posed in \cite{MaKa2022}, remains open. 
Motivated 	by these, in this article, we consider the spectral symmetry equivalence for the eccentricity matrices of  graphs in $\mathscr{B}$. 
Precisely, we prove that the eigenvalues of $E(G)$ are symmetric with respect to the origin if and only if $\diam(G)$ is odd, whenever $G\in \mathscr{B}$ with $\diam(G)\geq 4$.
By means of examples, we show that this equivalence does not hold for a general bi-block graph and for $G\in \mathscr{B}$ with $\diam(G)=3$.

To present another motivation for 
this paper, we now turn our attention to the inertias of  graph matrices (i.e., 
matrices that arise from graphs). Let us recall the notion of inertia.
Let $A$ be an $n \times n$ real symmetric matrix and let $i_+(A)$, 
$i_-(A)$  and  $i_0(A)$  denote the number of positive, 
negative and 
zero eigenvalues of $A$, respectively, including the multiplicities. The
\textit{inertia} of 
$A$ is the ordered triple $(i_+(A), i_{-}(A), i_0(A))$, and is denoted 
by 
$\In(A)$. It is  known that 
$i_+(A)+i_{-}(A)= \rank (A)$. An interesting and challenging problem in spectral graph theory 
is finding the eigenvalues and  inertias of  graph
matrices, see 
\cite{Bapa2010,BaMK2021,MGKA2020,MGRA2021,PaSP2021,WLBR2018,WLWL2020,WLLB2020} 
and the 
references  therein.
Among other results, it has been proved that the  inertia of the 
distance matrix of a tree on $n(\geq2)$ vertices is $(1,n-1,0)$, see 
\cite{Bapa2010}.
The  inertias of the eccentricity matrices of 
lollipop graphs,  coalescence of complete graphs,  coalescence of 
two cycles, trees, and clique trees have been computed 
in 
\cite{LiWB2022,MaKa2022,MGKA2020,PaSP2021}.
One of the objectives of this article is to  find the inertias 
of the eccentricity 	matrices of the graphs in $\mathscr{B}$. Some of the 
results obtained and the ideas used in this paper are similar to those in 
\cite{LiWB2022,MaKa2022} but the proofs  are different in many cases.

The outline of this article is as follows. In the next section, we recall 
some  results and notation used in this article. In Section 
\ref{tree}, 
we associate a tree $T_G$ for each $G\in \mathscr{B}$ and study their 
relationships where $\mathscr{B}$ is a subclass of bi-block graphs. Using 
these results, we derive the  centers of graphs in 
$\mathscr{B}$. Section 
\ref{inertia and spectral symmetry} deals 
with the inertia and the spectral symmetry of eccentricity matrix of $G\in 
\mathscr{B}$.  
Among other things, we give  an equivalent condition for the spectrum of 
$E(G)$  to be symmetric with respect to the origin where $G\in \mathscr{B}$.
Finally, we prove the irreducibility of  the eccentricity matrices 
of  graphs in $\mathscr{B}$.

\section{Preliminaries}\label{Prilims}
In this section, we recall  basic definitions and  notation that will 
be used in the sequel.

We 
denote the vertex set and the edge set of a graph $G$ by $V_G$ and $E_G$, 
respectively. 
A graph $G^{\prime}=(V_{G^{\prime}},E_{G^{\prime}})$ is said to be a 
\textit{subgraph} of 
$G$ if $V_{G^{\prime}}\subseteq V_G$ and $E_{G^{\prime}}\subseteq E_G$. A 
subgraph 
$G^{\prime}$ of $G$ is called an \textit{induced subgraph} of $G$ if 
the edges of $G^{\prime}$ are precisely the edges of $G$ whose 
ends are in $V_{G^{\prime}}$.
We 
denote the induced subgraph $G^{\prime}$  by  $G[V_{G^{\prime}}]$ and call it 
as the subgraph induced by  $V_{G^{\prime}}$.
A vertex $v\in V_G$ is said to be a \textit{cut-vertex} of $G$ if 
$G\setminus\{v\}$ is a disconnected graph. The notation $C_G$ stands for the 
collection of all cut-vertices of $G$.
A maximal connected subgraph of a graph $G$ is known as a \textit{component} of 
$G$.
The \textit{radius} and the 
\textit{diameter} of $G$ are, respectively, denoted by $\rad(G)$ and 
$\diam(G)$, 
and are defined by $\rad(G):=\min\{e(v): v\in V_G\}$ and $\diam(G):=\max\{e(v): 
v\in 
V_G\}$. A \textit{path} $P$ in $G$ is a subgraph of $G$ whose vertices 
are arranged in a sequence such that two vertices are adjacent in $P$ if and 
only if 
they are consecutive in the sequence. We denote a path of length $k$ between 
two vertices $a$ and $b$ in $G$ by $P_G(a,b)=au_1u_2\ldots u_{k-1}b$, where 
$u_i\in V_G$ for all $i$. We write the length of the path $P_G(a,b)$ by 
$l\left(P_G(a,b)\right)$.
A \textit{diametrical path} in $G$ is a 
shortest path between two vertices $u$ and $v$ such that
$d(u,v) = \diam(G)$. A vertex $v\in V_G$ is said to be a \textit{ central 
	vertex} if $e(v)=\rad(G)$. The \textit{center} of  $G$, denoted by 
$C(G)$, is the collection  of all central vertices of 
$G$.

A graph $G$ is said to be \textit{bipartite} if 
$V_G$ can be partitioned into two non-empty subsets $V_1$ and $V_2$ such 
that each edge of $G$ has one end in $V_1$ and the other end in $V_2$.
The pair $(V_1,V_2)$ is called a \textit{bipartition} of the bipartite 
graph $G$ and the sets $V_1$ and $V_2$ are referred to as the 
\textit{partite sets} of $G$.  A  bipartite graph  with  bipartition  
$(V_1,V_2)$ 
is 	said to be a complete bipartite graph if every vertex of $V_1$ is 
adjacent to all 
the vertices of $V_2$, and is denoted by $K_{\lvert V_1\rvert,\lvert V_2\rvert}$ where $\lvert V_1\rvert$ 
stands for  the
cardinality of  $V_1$.
For more details on  
graph-theoretic 
notions and terminologies, we refer to the book \cite{BaRa2012}.

Let $G\in \mathscr{B}$ and let $B$ be a block of $G$. Then $B$ is a complete 
bipartite 
graph $K_{l,m}$. Throughout 
this article, we assume that $\left(V_1(B), V_2(B)\right)$ is the bipartition 
of    $B$. The block $B$ of $G$ is said to be a \textit{bridge block} if 
$\lvert V_B\cap C_G\rvert=2$, and a
\textit{leaf block } if $\lvert V_B\cap C_G \rvert=1$.

Let $A$ be an $m\times n$  matrix. We write the transpose of $A$, $i$-th row of $A$, $i$-th column of $A$ and the rank 
of $A$ by $A^{\prime}$, $A_{i*}$, $A_{*i}$ and $\rank(A)$,
respectively. 
We denote the  
principal submatrix of $A$ whose rows and 
columns are  indexed, respectively,  by the sets $U\subseteq\{1,2,\ldots,m\}$ and 
$V\subseteq\{1,2,\ldots,n\}$ by $A\left([U\mid V]\right)$. 
The notations $J$ and $O$ 
are used to denote the matrices with all elements equal to $1$ and $0$, 
respectively, and the orders of the matrices are clear from the context. The 
determinant of a square matrix $A$ is written by $\det(A)$.

In the following, we collect  some known 
results which are needed in this paper.
\begin{theorem}[{\cite{Zhan2011}}]\label{Schurcomplement_det_formula}
	Let  $A$ and $D$ be $r\times r$ and $s\times  s$ real matrices, respectively, 
	and 
	let $M=\left(\begin{smallmatrix}
		A & B\\[3pt]
		C & D\\
	\end{smallmatrix}\right)$ be 
	a symmetric partitioned matrix of order $n$. If $A$ is nonsingular, then 
	\begin{itemize}
		\item [(i)] $ \det(M)= \det (A)\det(D-CA^{-1}B)$.
		\item [(ii)] $ \In (M)= \In(A)+\In(D-CA^{-1}B)$.
	\end{itemize}
\end{theorem}
In fact, the above result holds for any non-singular principal submatrix $A$ of 
$M$.

\begin{theorem}[{\cite[P.25]{Zhan2011}}]\label{Co-eff_char_poly}
	If $M$ is an $n \times n$ real matrix, then the characteristic polynomial 
	of $M$ is given by
	$\chi(x)=x^n-\delta_1{x}^{n-1}+\delta_2{x}^{n-2}+\ldots
	+(-1)^{n-1}\delta_{n-1}x+ (-1)^{n}\delta_{n}$
	where $\delta_r$ denotes the sum of all principal minors of order $r$ for 
	all $r=1,2,\ldots,n$.
\end{theorem}
\begin{theorem}[\hspace{1sp}\cite{Zhan2011}]\label{Interlacing theorem}
	Let $M$ be a symmetric matrix of order $n$ and $A$ be a  principal submatrix of 
	$M$ order $m$ where $1\leq m \leq n$. If the 
	eigenvalues of $M$ and $A$ are
	$\lambda_1\geq \lambda_2\geq \cdots \geq\lambda_n$ and $\beta_1\geq 
	\beta_2\geq \cdots \geq
	\beta_m$, respectively, then
	$\lambda_i\geq \beta_i \geq \lambda_{n-m+i},~\text{for all 
		$i=1,2,\ldots,m$}$.
Moreover, $	i_+(M)\geq i_+(A)~ \text{and}~i_-(M)\geq i_-(A).$
\end{theorem}

\begin{lemma}[see 
\cite{Bapa2010,MaKa2022}]\label{symmetric_about_origin_equiv_lemma}
Suppose that $p(x) = 
{x}^n+\alpha_1{x}^{n-1}+\alpha_2{x}^{n-2}+\hdots+\alpha_{n-1}x+\alpha_n$ 
is a  polynomial whose roots are all non-zero real numbers. 
If there exists $i \in \{1, \dots, 	n-1\}$ such that  $\alpha_i \neq 0$ and 
$\alpha_{i+1}\neq0$, then the roots 
of $p(x)$ are not symmetric with respect to the origin (i.e., $p(a) 
=0$ but 
$p(-a)\neq 0$ for some real number $a$).
\end{lemma}

\section{Relations between $G\in \mathscr{B}$ and its associated tree 
$T_G$}\label{tree}

In this section, we associate a tree $T_G$ for each graph  $G\in \mathscr{B}$ and 
obtain some interconnections between $G$ and $T_G$ by employing the
properties of trees. 
In particular, we  show that 
$\diam(G)=\diam(T_G)$ and  $C(T_G)\subseteq C(G)$. Making use of these 
relations, we explicitly find the centers  of  graphs in $\mathscr{B}$. These results  will 
be used in the next section to study the inertia and the spectral symmetry of the  eccentricity 
matrix of a  graph in $\mathscr{B}$.

The idea of analyzing a  graph  $G\in \mathscr{B}$ through the  associated  tree $T_G$ is motivated by \cite{LiWB2022,MaKa2022}. In \cite{LiWB2022}, the 
authors  studied  the eccentricity matrices of clique trees by constructing 
trees of particular types. While defining trees, they considered  
non-cut-vertices only from the leaf blocks of  clique trees. In our case, we 
have to include non-cut-vertices  from both  leaf blocks and bridge blocks of  
$G\in \mathscr{B}$ in order to obtain a tree satisfying specific properties and 
thus the construction is different from \cite{LiWB2022}. 

Let us begin this section by associating a subgraph $T_G$ for each $ 
G\in \mathscr{B}$. It will be shown later that $T_G$ is a tree.
We  collect some selected 
vertices 
from  each block $B$ of $G$ to define $T_G$. If $B= K_{1,1}$, then we collect both the vertices 
of $B$. Suppose that $B\neq K_{1,1}$ and  $B$ is a  bridge block  with $V_B\cap 
C_G=\{z_1,z_2\}$.  If both the 
cut-vertices $z_1$ and $z_2$ lie on the same partite set, say 
$V_1(B)$,  
then we choose three vertices from $V_B$, which are  $z_1$, $z_2$ and  
a  non-cut-vertex in  
$V_2(B)$ with the minimum vertex label; otherwise, we collect 
exactly two vertices  $z_1$ and  $z_2$ from $V_B$.
In the case of leaf block 
$B\neq K_{1,1}$, we 
collect exactly three vertices, which are the cut-vertex in $V_B$, and  
non-cut-vertices with the minimum vertex label in each partite sets 
$V_1(B)$ and $V_2(B)$. 
The graph $T_G$ is 
defined as the 
subgraph induced by the vertex set which is the union of all selected vertices in each block of $G$. The 
precise construction of $T_G$ is given below.

\begin{definition}
Let $ G\in \mathscr{B}$ and let $B$ be a block of $G$ with 
bipartition $(V_1(B), V_2(B))$. For $i =1, 2$, we denote the 
non-cut-vertex  with the minimum vertex	label in $V_i(B)$, if it exists, 
by $u_i$. If $B\neq  K_{1,1}$, we define
\begin{equation*}
	NC_B:= \begin{cases} 
		\{u_2\} & \text{if 	$\lvert V_1(B)\cap C_G\rvert =2$},\\
		\{u_1\} & \text{if $\lvert V_2(B)\cap C_G\rvert =2$}, \\
		\{u_1,u_2\}& \text{if $B$ is a leaf block},\\
		\emptyset& \text{otherwise}.
	\end{cases}
\end{equation*}
If  $B= K_{1,1}$, then we set $NC_B:=	\{u_1,u_2\}\cap V_B$.
Let $\displaystyle S:=\cup_{B}	NC_B$ 
where $B$ runs over all blocks  of 
$G$. 
Define $T_G$ as the subgraph induced by the vertex subset $S\cup C_G$. That is, 
\begin{equation}\label{T_G}
	T_G:=G\left[S\cup C_G\right].
\end{equation}
\end{definition}
\begin{remark}
We mention that the choice of a non-cut-vertex with the minimum vertex label 
ensures the  uniqueness of  $T_G$.
\end{remark}
\begin{example}\label{Illustrating_eg}
We illustrate  the construction of $T_G$ for a given graph $G\in 
\mathscr{B}$ through the following figures. The graph $G$ has three leaf blocks ($K_{1,1},K_{1,1}$ and $K_{2,3}$) and three bridge blocks ($K_{2,2},K_{1,1}$ and $K_{3,3}$).
\begin{figure}[ht]
	\begin{subfigure}[b]{0.5\linewidth}
		\centering
		\begin{tikzpicture}[scale=1]
			\draw[fill=black] (-3,2) circle (2pt);
			\draw[fill=black] (-3,1) circle (2pt);
			\draw[fill=black] (-2,2) circle (2pt);
			\draw[fill=black] (-2,1) circle (2pt);
			\draw[fill=black] (-1,2) circle (2pt);
			\draw[fill=black] (-1,1) circle (2pt);
			\draw[fill=black] (-1,0) circle (2pt);
			\draw[fill=black] (-1,-1) circle (2pt);
			\draw[fill=black] (0,2) circle (2pt);
			\draw[fill=black] (0,1) circle (2pt);
			\draw[fill=black] (0,-1) circle (2pt);
			\draw[fill=black] (0,-2) circle (2pt);
			\draw[fill=black] (1,2) circle (2pt);
			\draw[fill=black] (1,1) circle (2pt);
			\draw[fill=black] (1,0) circle (2pt);
			\draw[fill=black] (0,0) circle (2pt);

			\draw[thin] (-2,2)--(-3,1);
			\draw[thin] (-2,2)--(-3,2);
			
			\draw[thin] (-2,2)--(-1,2);
			\draw[thin] (-2,2)--(-1,1);
			\draw[thin] (-2,1)--(-1,2);
			\draw[thin] (-2,1)--(-1,1);
			
			\draw[thin] (0,2)--(1,2);
			\draw[thin] (0,2)--(1,1);
			\draw[thin] (0,2)--(1,0);
			\draw[thin] (0,1)--(1,2);
			\draw[thin] (0,1)--(1,1);
			\draw[thin] (0,1)--(1,0);
			\draw[thin] (0,0)--(1,2);
			\draw[thin] (0,0)--(1,1);
			\draw[thin] (0,0)--(1,0);
			
			\draw[thin] (-1,2)--(0,2);

			\draw[thin] (-1,0)--(0,0);
			\draw[thin] (-1,0)--(0,-1);
			\draw[thin] (-1,0)--(0,-2);
			\draw[thin] (-1,-1)--(0,0);
			\draw[thin] (-1,-1)--(0,-1);
			\draw[thin] (-1,-1)--(0,-2);
			
			\node at (-3,.7) {$v_2$};
			\node at (-2,.7) {$v_4$};
			\node at (-1,.7) {$v_6$};
			\node at (-1,1.7) {$v_5$};
			\node at (-2,1.7) {$v_3$};
			\node at (-3,1.7) {$v_1$};
			
			\node at (-.1,1.7) {$v_7$};
			\node at (1.2,1.7) {$v_{10}$};
			\node at (1.1,0.7) {$v_{11}$};
			\node at (1.1,-.3) {$v_{12}$};
			\node at (-.2,.8) {$v_{8}$};
			\node at (-.15,.2) {$v_{9}$};
			\node at (.35,-1) {$v_{15}$};
			\node at (-1.4,-1) {$v_{14}$};
			\node at (-1.4,0) {$v_{13}$};
			\node at (0.35,-2) {$v_{16}$};
		\end{tikzpicture}
		\caption{A graph $G\in \mathscr{B}$} 
\end{subfigure}
\begin{subfigure}[b]{0.5\linewidth}
	\centering
	\begin{tikzpicture}[scale=1]
		\draw[fill=black] (-3,2) circle (2pt);
		\draw[fill=black] (-3,1) circle (2pt);
		\draw[fill=black] (-2,2) circle (2pt);
		\draw[fill=black] (-1,2) circle (2pt);
		\draw[fill=black] (-1,0) circle (2pt);
		\draw[fill=black] (0,2) circle (2pt);
		\draw[fill=black] (0,-1) circle (2pt);
		\draw[fill=black] (1,2) circle (2pt);
		\draw[fill=black] (0,0) circle (2pt);

		\draw[thin]  (-2,2)--(-3,1);
		\draw[thin]  (-2,2)--(-3,2);
		
		\draw[thin] (-2,2)--(-1,2);
		
		\draw[thin] (0,2)--(1,2);
		\draw[thin] (0,0)--(1,2);

		\draw[thin]  (-1,2)--(0,2);

		\draw[thin] (-1,0)--(0,0);
		\draw[thin] (-1,0)--(0,-1);

		\node at (-3,.7) {$v_2$};
		\node at (-1,1.7) {$v_5$};
		\node at (-2,1.7) {$v_3$};
		\node at (-3,1.7) {$v_1$};
		
		\node at (-.1,1.7) {$v_7$};
		\node at (1.2,1.7) {$v_{10}$};
		\node at (-.15,.2) {$v_{9}$};
		\node at (.35,-1) {$v_{15}$};
		\node at (-1.4,0) {$v_{13}$};
		
	\end{tikzpicture}
	\caption{The tree $T_G$ associated with $G$} 
\end{subfigure}
\end{figure}
\end{example}
\vspace{-.65cm}
In the following, we  
mention a few properties of  $G\in \mathscr{B}$ that are needed in the 
sequel. For more details, we refer to  \cite{BaRa2012,Dies2010}.
\begin{itemize}
\item [$(P_1)$] Suppose that $B$ is a block of $G$ with bipartition 
$(V_1(B),V_2(B))$. Then $\lvert V_1(B)\rvert=1$ if and only if $\lvert V_2(B)\rvert =1$.
\item [$(P_2)$] Any shortest path in $G$ can contain at most three consecutive 
vertices from a single block $B$ of $G$.
\item [$(P_3)$] Cycles of $G$ are exactly cycles of its blocks. 
\item [$(P_4)$] Let $v\in V_G$. Then $v$ is a cut-vertex of $G$ if and only if 
it lies in at least two blocks of $G$.
\end{itemize}
Using the property $(P_4)$ and the fact that  each edge of $G$ lies on exactly 
one block of 
$G$  (P. $59$ in \cite{BaRa2012}), we obtain the following remark.
\begin{remark}\label{p6}
Let $G\in \mathscr{B}$ and   $B$ be a block of $G$. If $xy$ 
and 
$yz$ are edges in $G$ such that 
$x,y\in V_B$ and $z\not \in V_B$, then $y\in C_G$.
Moreover, if
$P_G(u_0,u_{k+1})=u_0u_1\ldots u_ku_{k+1}$ $(k\geq 2)$ is a shortest 
path between $u_0$ 
and $u_{k+1}$ in $G$ such that  the edge 
$u_{i-1}u_{i}$ 
belongs to the block $B$  for some $1 \leq 
i \leq k$  and $u_{i+1}\not \in V_B$, then $u_i\in 
C_G$. 
That is, a shortest path leaves a block $B$ of $G$ and enters into another 
block $B^{\prime}$ of $G$
through a cut-vertex of $G$ which lies in $ V_B\cap  V_{B^{\prime}}$. 
\end{remark}
Many results of this section deals  the relationship between 
$G\in \mathscr{B}$  and its associated graph $T_G$. We first show that the 
graph $T_G$ is a tree.
In order to distinguish, we use the notations $d_G(a,b)$ and  
$d_{T_G}(a,b)$, respectively, to denote the distance between the
vertices $a$ and $b$ 
with 
respect to 
$G$, and with 
respect to $T_G$.
\begin{lemma}\label{T_G tree}
Let $G\in \mathscr{B}$ and  $T_G$  be the subgraph of $G$ defined in 
(\ref{T_G}). Then $T_G$   is a tree and  $d_G(a,b)=d_{T_G}(a,b)$ for all $a,b 
\in V_{T_G}$.
\end{lemma}
\vspace{-.5cm}
\begin{proof}
Note that each block of $G$ is a complete bipartite graph.
Since $T_G$ has at most three vertices from each block of $G$, by 
$(P_3)$,  $T_G$ 
contains no cycle. We claim that $T_G$ is connected.
Let  $a,b \in V_{T_G}$. Since $G$ is connected, there is a path
between the vertices $a$ and $b$ in $G$. Let
$P_G(a,b)=au_1u_2\ldots u_{k}b$ be a shortest path in 
$G$. If $u_i \in V_{T_G}$ for all 
$i=1,2,\ldots,k$ then the claim follows. 
Suppose that $u_j \not \in V_{T_G}$ for some $j \in \{1,2,\ldots,k\}$. Let 
$j_0$ be the smallest index such that $u_{j_0} \not \in V_{T_G}$. Then  
$u_{j_0}$ is not a cut-vertex of $G$ because all the cut-vertices of $G$ 
belong to $T_G$. Assume that the edge $u_{j_0-1}u_{j_0}$ belongs to the 
block $B$ of $G$. If $j_0=1$ then take  $u_0=a$. Without loss of generality, we
assume that $u_{j_0} \in V_1(B)$. Then $u_{j_0-1} \in V_2(B)$. Since 
$u_{j_0}\not \in C_G$, by Remark \ref{p6},  $ u_{j_0+1}\in 
V_2(B)$  where we take $u_{j_0+1}=b$ if $j_0=k$.
By the construction of $T_G$,  it is clear that $ V_1(B)\cap V_{T_G}\neq 
\emptyset$. Choose $u_{j_0}^{\prime} \in 
V_1(B)\cap V_{T_G}$. We now show that $u_{j_0}^{\prime} \not \in P_G(a,b)$. 
Clearly, $u_{j_0}^{\prime} \neq u_{j_0-1}$ and   $u_{j_0}^{\prime} \neq 
u_{j_0+1}$. If 
there exists $i\in 
\{1,2,\ldots,j_0-2\}$ such that    $u_{j_0}^{\prime} = u_i$ then $au_1\ldots 
u_{i}u_{j_0+1}\ldots b$ is a path between $a$ and $b$ whose length is 
strictly less than $l\left(P_G(a,b)\right)$ which is not possible. Therefore
$u_{j_0}^{\prime} \neq u_i$ for 
all $i\in 
\{1,2,\ldots,j_0-2\}$. Similarly, we see that  $u_{j_0}^{\prime} \neq u_i$ for 
all $i\in  \{j_0+2, \ldots,k\}$. Hence $u_{j_0}^{\prime} \not \in P_G(a,b)$. 
We obtain a new path
$P^{\prime}_G(a,b)$ from  $P_G(a,b)$ 
by replacing the edges $u_{j_0-1}u_{j_0}$ and $u_{j_0}u_{j_0+1}$, 
respectively, by $u_{j_0-1}u_{j_0}^{\prime}$ and $u_{j_0}^{\prime}u_{j_0+1}$. 
Clearly, $l( P^{\prime}_G(a,b))=l( P_G(a,b))$. 
We replace $P_G(a,b)$ by $P^{\prime}_G(a,b)$. 
Repeating the above argument leads to obtain 
a path $ P_{T_G}(a,b)$ in $T_G$ such that
$l(P_{T_G}(a,b))=l( P_G(a,b))$. This implies that $T_G$ is connected and 
$d_{T_G}(a,b)\leq l(P_{T_G}(a,b))=l( P_G(a,b))=d_G(a,b)$. Since $T_G$ is a 
subgraph of $G$, we have $d_{T_G}(a,b)\geq d_G(a,b)$. Hence  $d_{T_G}(a,b)= 
d_G(a,b)$. 
\end{proof}
The following lemma is useful in establishing the result $\diam(G)=\diam(T_G)$. For  notational simplicity, we also use the notation  $d(a,b)$ to denote the 
distance 
between two vertices $a$ and $b$ in $G$ instead of $d_{G}(a,b)$.
\begin{lemma}\label{distance same}
Let  $G\in \mathscr{B}$. Given  $a$ and $b$ in $V_G$, there exist  
$a^{\prime}$ and $b^{\prime}$ in  $V_{T_G}$ such that 
$d(a,b)=d(a^{\prime},b^{\prime})$.
\end{lemma}
\vspace{-.5cm}
\begin{proof}
Let $a,b \in G$.  If $a,b\in V_{T_G}$ then the result follows. 
Consider the case $b\not \in V_{T_G}$. We claim 
that there exists $b^{\prime}\in V_{T_G}$ such that 
$d(a,b)=d(a,b^{\prime})$. Let $P_G(a,b)=au_1u_2\ldots u_{k}b$ be a path 
such that
$d(a,b)=l\left(P_G(a,b)\right)$ where 
$u_{k}\in V_1(B)$ and $b\in V_2(B)$ for  some block $B$ of $G$. We first  
prove the claim for 	$k\geq 2$.

\noindent	\textbf{Case 1}: Suppose that $u_{k} \not \in C_G$. Then by Remark 
\ref{p6},
$u_{k-1}\in V_2(B)$, and by $(P_2)$, $u_{k-2} \not \in  V_B$ where we take 
$u_{k-2}=a$ if $k=2$. Therefore, by Remark \ref{p6},  $u_{k-1} \in C_G$.\\
\textbf{Subcase 1.1}: 
Let $B$ be a leaf block. By the construction of $T_G$,   there exists a 
non-cut-vertex 		$b^\prime \in V_2(B)\cap 	V_{T_G}$. Using $(P_3)$ and  
$k\geq 2$, we see that $a \not \in V_B$ and  $b^\prime 	\not \in P_G(a,b)$. 
Therefore, $b^{\prime} u_{k} u_{k-1}\ldots u_1a$ is a 
path. It is clear that $u_{k-1}\in \bar{P}_G(b^{\prime},a)$ for all 
shortest paths 
$ \bar{P}_G(b^{\prime},a)$ between the vertices $b^{\prime}$ and $a$. Hence 
$d(b^{\prime},a)=d(b^{\prime},u_{k-1})+d(u_{k-1},a)=d(b,a)$.\\
\textbf{Subcase 1.2}: 
Assume that $B$ is a bridge block and  $z$ is a cut-vertex of $G$ in $B$ 
other   than $u_{k-1}$.  \\
\textbf{Subcase 1.2.1}: 	If $z \in V_2(B)$ then, by $(P_3)$, $z \not 
\in P_G(a,b)$. Therefore, $P_G(z,a)=zu_{k} u_{k-1}\ldots u_1a$ is a 
path. Again  by $(P_3)$, $u_{k-1}\in \bar{P}_G(z,a)$  for all shortest 
paths $ 
\bar{P}_G(z,a)$ between $z$ and $a$. This fact leads to
$d(z,a)=d(z,u_{k-1})+d(u_{k-1},a)=d(b,a)$. Since $z\in 
C_G$, we have $z\in V_{T_G}$. Therefore, in this case, choose 
$b^{\prime}=z$.\\		
\textbf{Subcase 1.2.2}: Suppose that $z \in V_1(B)$. Then $z\neq u_{k-2}$ 
because $u_{k-2}\not \in V_B$. So, by $(P_3)$, $z \not 
\in P_G(a,b)$,  and $u_{k-1}\in \bar{P}_G(z,a)$  for all shortest paths 
$ 
\bar{P}_G(z,a)$ between $z$ and $a$. This implies that 
$d(z,a)=1+d(u_{k-1},a)=d(b,a)-1$. 
Since $z\in C_G\cap V_1(B)$, by $(P_4)$, $z\in V_{B^{\prime}}$ for some 
block	$B^{\prime}\neq B$. By the construction of $T_G$, there exists 
$b^{\prime} \in V_{B^{\prime}}\cap V_{T_G}$ such that $b^{\prime}$ is 
adjacent to $z$. Again by $(P_3)$, $b^{\prime} \not \in 
P(z,a)=zu_{k-1}\ldots u_1a$ and $z\in  \bar{P}_G(b^{\prime},a)$  for 
all shortest paths $\bar{P}_G(b^{\prime},a)$ between $b^{\prime}$ and 
$a$. 
Therefore, 
$d(b^{\prime},a)=d(b^{\prime},z)+d(z,a)=1+d(z,a)=d(b,a)$.\\
\noindent	\textbf{Case 2}: Let $u_{k} \in C_G$. Since $b\in V_2(B)$ 
with $b \not\in V_{T_G}$, we have $\lvert V_2(B)\rvert\geq 2$ because $V_2(B)\cap 
V_{T_G}\neq \emptyset$. Therefore, $B\neq 
K_{1,1}$.\\
\textbf{Subcase 2.1}:
If $B$ is a leaf block, then choose $b^{\prime}\in V_2(B)\cap V_{T_G}$. We 
show that  $u_{k-1} \not \in V_B$. On the contrary, assume that
$u_{k-1}  \in V_B$. Then, by $(P_2)$, $u_{k-2}\not \in V_B$. Also, by 
Remark \ref{p6}, $u_{k-1}\in C_G$ which is not possible because $B$ is a 
leaf 
block with  $u_k\in V_B\cap C_G$. Hence $u_{k-1} \not \in V_B$ and this 
yields 
$b^{\prime}\neq u_{k-1}$. Therefore, by ($P_3$), we observe that 
$b^{\prime}\not \in 
P_G(a,b)$, and by 	Remark \ref{p6}, $u_k\in\bar{P}_G(b^{\prime},a)$ for 
all shortest  paths 
$\bar{P}_G(b^{\prime},a)$ between $b^{\prime}$ and $a$. Clearly, 
$b^{\prime}u_ku_{k-1}\ldots u_1a$ is a 
path, and we have $d(b^{\prime},a)=d(b^{\prime}, u_k)+d(u_k,a)=d(b,a)$.\\
\textbf{Subcase 2.2}: Suppose that $B$ is a bridge 
block.  Let $z\in V_B\cap C_G$ with $z\neq u_k$.\\
\textbf{Subcase 2.2.1}: If $z \in V_1(B)$ then take $b^{\prime}\in 
V_2(B)\cap V_{T_G}$. Rest of the proof in this case is same as that of 
subcase $2.1$.\\		
\textbf{Subcase 2.2.2}: Let $z \in V_2(B)$.
If $z\neq u_{k-1}$, then $u_{k-1} \not \in V_B$. By $(P_3)$, $z \not 
\in P_G(a,b)$, and $u_{k}\in \bar{P}_G(z,a)$  for all shortest paths $ 
\bar{P}_G(z,a)$ between $z$ and $a$. Therefore, $zu_ku_{k-1}\ldots 
u_1a$ is a shortest path 
and hence $d(z,a)=d(b,a)$.
Suppose that $z= u_{k-1}$. Since $u_k \in C_G$, there exists $b^{\prime}\in 
V_{B^{\prime}}\cap V_{T_G}$  such that $b^{\prime}$ 
is adjacent to $u_k$ where $B^{\prime}$ is a block of $G$ different from 
$B$. Then, by $(P_3)$, $b^{\prime} \not 
\in P_G(u_k,a)$, and  $u_{k}\in \bar{P}_G(b^{\prime},a)$  for all shortest 
paths $	\bar{P}_G(b^{\prime},a)$. 
By considering the 
path 
$P_G(b^{\prime},a)=b^{\prime}u_ku_{k-1}\ldots u_1a$, we obtain 
$d(b^{\prime},a)=d(b,a)$.\\
Hence, from the above cases, we have 
$d(a,b)=d(a,b^{\prime})$ for some $b^{\prime}\in V_{T_G}$.  The proof for the 
case   $k=1$ can be verified similarly. If $a\in V_{T_G}$ 
then the result follows. If $a \not\in V_{T_G}$ then repeat the above argument 
to the newly obtained path $P_G(b^{\prime},a)$ to find   $a^{\prime}\in 
V_{T_G}$ such that $d(a,b)=d(a^{\prime},b^{\prime})$. This completes the proof.
\end{proof}

The next two lemmas are the consequences of the above result.

\begin{lemma}\label{diam of G and T_G}
For each $G\in \mathscr{B}$, $\diam(G)=\diam(T_G)$. 
\end{lemma}
\vspace{-.5cm}
\begin{proof}
Let	$P_G(a,b)$ be a diametrical path in $G$. By 
Lemma 
\ref{distance same}, $d_G(a,b)=d_G(a^{\prime},b^{\prime})$ for some 
$a^{\prime},b^{\prime} \in V_{T_G}$. Therefore, by Lemma \ref{T_G tree},  
$d_G(a^{\prime},b^{\prime})=d_{T_G}(a^{\prime},b^{\prime})$. This implies 
that $\diam(G)=d_G(a,b)=d_{T_G}(a^{\prime},b^{\prime})\leq \diam(T_G)$. 
To prove the other inequality, let $x,y \in 
V_{T_G}$ such that $d_{T_G}(x,y)= \diam(T_G)$. Using Lemma \ref{T_G tree}, 
we get $d_{T_G}(x,y)= d_{G}(x,y)\leq\diam(G)$.
\end{proof}

Let $H$ be a connected graph and let $u\in V_H$. We denote the  
eccentricity of a vertex $u$ 
with respect 
to $H$ by $e_H(u)$, and the subscript is omitted  if it is 
clear from the context.

\begin{lemma}\label{ecc_w.r.t T_G}
If $G\in \mathscr{B}$, then $e_G(a)=e_{T_G}(a)$ for all $a\in 
V_{T_G}$.
\end{lemma}
\vspace{-.5cm}
\begin{proof}
Let $a\in V_{T_G}$. By Lemma \ref{T_G tree},  $\{d_{T_G}(a,x): x\in 
V_{T_G}\} \subseteq \{d_G(a,x): x\in V_{G}\}$. This implies
$e_{T_G}(a)\leq e_G(a)$. Let $b \in V_G$ be such that
$e_G(a)=d_G(a,b)$. If $b\in 
V_{T_G}$ then again by Lemma \ref{T_G tree}, $d_G(a,b)=d_{T_G}(a,b)\leq 
e_{T_G}(a)$. Suppose that  $b\not \in V_{T_G}$. Then from the 
proof of 
Lemma \ref{distance same}, we have $d_G(a,b)=d_G(a,b^{\prime})$ for 
some 
$b^{\prime}\in V_{T_G}$. So, 	$e_{T_G}(a)\geq e_G(a)$ follows by 
Lemma \ref{T_G tree}.
\end{proof}

Since most of the results of this paper hold for graphs in 
$\mathscr{B}$ with diameters at least four (for instance, see Example 
\ref{counter ex}), we  deal only with 
those graphs hereafter. 
In the following lemma, the eccentricity of a non-cut-vertex 
(when it exists) of $G\in \mathscr{B}$ is presented in terms of the  
eccentricities of the cut-vertices of $G$. In Section \ref{inertia and spectral 
symmetry}, we see  the usefulness of 
this lemma in 
proving the spectral symmetry of $E(G)$.
\begin{lemma}\label{eccentricity lemma}
Let $G\in \mathscr{B}$ be such that $\diam(G) \geq 4$. 
Let $B$ be a block of $G$ and let $u\in V_B$ be a non-cut-vertex of $G$. 
\begin{itemize}
\item [(i)] Suppose that $B$ is a bridge block of $G$ and  $v_1$ and $v_2$ 
are cut-vertices of $G$ belong to $B$. If $e(v_1)\leq e(v_2)$, then 
\begin{eqnarray*}\label{Eccentricty bridge}
e(u)= \begin{cases}  
	e(v_2) & \text{if } v_1,v_2,u\in V_1(B),\\
	e(v_2)-1 & \text{if } v_1,v_2\in V_1(B)~\text{and}~ u\in V_2(B),\\
	e(v_2)+1 & \text{if } v_1,u\in V_1(B)~\text{and}~ v_2\in V_2(B),\\
	e(v_1)+1 & \text{if } v_1\in V_1(B)~\text{and}~ v_2,u\in V_2(B).
\end{cases}
\end{eqnarray*}

\item [(ii)] If $B$ is a leaf block  and $v$ is a 
cutvertex of $G$ belongs to $B$, then 
\begin{eqnarray*}
e(u)= \begin{cases} 
	e(v)+2 & \text{if }
	u,v\in V_1(B)~\text{or}~ u,v\in V_2(B),\\
	e(v)+1 & \text{otherwise}.
\end{cases}
\end{eqnarray*}
\end{itemize}
\end{lemma}
\vspace{-.5cm}
\begin{proof}
$(i)$	 Since $B$ is a bridge block with $v_1,v_2\in 
V_B\cap C_G$, 
there exists a vertex
$x\in V_G\setminus V_B$ such that $e(v_2)=d(v_2,x)$. Let $P(v_2,x)$ be 
a shortest 	path between $v_2$ and $x$ in $G$. We claim that 
$v_1\in P(v_2,x)$. On the contrary, assume that $v_1 \not \in 
P(v_2,x)$. 
Then, by  $(P_3)$ and 	Remark \ref{p6}, $v_2\in \bar{P}(v_1,x)$ 
for 
all paths 
$\bar{P}(v_1,x)$. 
Therefore, 		
$e(v_1)\geq d(v_1,x)=d(v_1,v_2)+d(v_2,x)\geq 1+e(v_2)>e(v_2)$, 
which 
contradicts the assumption $e(v_1)\leq e(v_2)$. Hence the claim $v_1\in 
P(v_2,x)$ follows. This implies that $d(x,v_2)=d(x,v_1)+d(v_1,v_2)$. 
Let $u\in V_B\setminus C_G$. Note that $v_1 \in 	\bar{P}(x,u)$ for all 
shortest paths $\bar{P}(x,u)$ by $(P_3)$ and Remark \ref{p6}. Therefore,
$d(x,u)=d(x,v_1)+d(v_1,u)=[d(x,v_2)-d(v_1,v_2)]+d(v_1,u)$, and hence 
\begin{equation}\label{geq ineq}
e(u)\geq 
d(x,u)=d(x,v_2)+[d(v_1,u)-d(v_1,v_2)]=e(v_2)+[d(v_1,u)-d(v_1,v_2)].
\end{equation}	
Since $v_1,v_2\in 
V_B\cap C_G$,  
there always exists 
$y\in V_G\setminus V_B$ such that  $e(u)=d(u,y)$. Let $P(u,y)$ be a 
shortest path between $u$ and $y$.
Suppose that $v_1 \in P(u,y)$ and $v_2 \not \in P(u,y)$. Then, by ($P_3$) 
and Remark \ref{p6}, 
any shortest path 
$\bar{P}(v_2,y)$ contains $v_1$. This implies that 
$d(v_2,y)=d(v_2,v_1)+d(v_1,y).$
Therefore, we have 
\begin{align}\label{leq ineq A}
e(u)=	d(u,y)&=d(u,v_1)+d(v_1,y)\\\notag
&=d(u,v_1)+[d(v_2,y)-d(v_1,v_2)]\\\label{leq ineq B}
& \leq 	 	e(v_2)+[d(u,v_1)-d(v_1,v_2)].
\end{align}
If $v_1 \not \in  P(u,y)$,  then by Remark \ref{p6}, $v_2 \in  P(u,y)$. 
Also, by 
$(P_3)$, $v_2\in 	\bar{P}(v_1,y)$ for all shortest paths 
$	\bar{P}(v_1,y)$. Then $d(y,v_1)=d(y,v_2)+d(v_2,v_1)$. Hence
\begin{align}\label{leq ineq C2}
e(u)=	d(u,y)=d(u,v_2)+d(v_2,y)&=d(u,v_2)+[d(y,v_1)-d(v_1,v_2)]\\
\label{leq ineq C1}
& \leq 	 	e(v_1)+[d(u,v_2)-d(v_1,v_2)]\\\label{leq ineq C}
& \leq 	 	e(v_2)+[d(u,v_2)-d(v_1,v_2)].
\end{align}
\noindent	\textbf{Case 1}: Let $v_1,v_2\in V_1(B)$. Since $P(u,y)$ is a 
shortest path between $u$ and $y$, by Remark \ref{p6}, we have
$\lvert V_{P(u,y)}\cap\{v_1,v_2\}\rvert =1$ where $V_{P(u,y)}$ is the set of all
vertices in the path $P(u,y)$. Now, the result follows from (\ref{geq 
ineq}), 
(\ref{leq ineq B}) and (\ref{leq ineq C}), case-by-case.\\
\noindent	\textbf{Case 2}: Assume that $v_1 \in V_1(B)$ and $v_2 \in 
V_2(B)$.\\
\textbf{Subcase 2.1}: Let $u\in V_1(B)$. From (\ref{geq ineq}),
$e(u)\geq 	e(v_2)+1$. To prove the other 	inequality,  we first show 
that 	$v_1\in P(y,u)$. On the 
contrary, assume that $v_1 \not \in P(y,u)$. Then by Remark \ref{p6}, $v_2 
\in 
P(y,u)$ and so, by  $(P_3)$, $v_2 \in 	\bar{P}(y,v_1)$ for all shortest 
paths 
$	\bar{P}(y,v_1)$.
Therefore, 
$e(u)=d(y,u)=d(y,v_2)+d(v_2,u)=d(y,v_2)+d(v_2,v_1)=d(y,v_1)\leq e(v_1)$. 
That is, $e(v_1)\geq e(u)\geq e(v_2)+1$  where the last 
inequality follows from 
(\ref{geq ineq}). This contradicts the assumption $e(v_1)\leq e(v_2)$. 
Hence 	$v_1\in P(y,u)$. 
If $v_1\in P(y,u)$ and  $v_2 \not \in P(y,u)$ then by (\ref{leq 
ineq B}), $e(u)\leq e(v_2)+1$. Suppose 
that $v_1,v_2 \in P(y,u)$. Then 
$e(u)=d(y,u)=d(y,v_2)+d(v_2,u)\leq e(v_2)+1$. 	 This 
completes the proof in this case.\\
\textbf{Subcase 2.2}:  Suppose that $u\in V_2(B)$. Let $e(v_1)=d(z,v_1)$ 
for some $z \in V_G\setminus V_B$
and 
let 
$P(z,v_1)$ be a shortest path between $z$ and $v_1$ in $G$. If $v_2\in 
P(z,v_1)$ then  by $(P_3)$, $v_2\in 
\bar{P}(z,u)$ for all shortest paths $	\bar{P}(z,u)$. 
We have  
$d(z,u)=d(z,v_2)+d(v_2,u)$ and $d(z,v_1)=d(z,v_2)+1$. This implies that  
$e(u)\geq 
d(z,u)=d(z,v_1)+1=e(v_1)+1$.
Suppose that $v_2\not \in P(z,v_1)$. Then again by   $(P_3)$, $v_2\not\in 
\bar{P}(z,u)$, and  by   Remark \ref{p6}, $v_1\in 
\bar{P}(z,u)$ for all paths $\bar{P}(z,u)$. 
So, 
$d(z,u)=d(z,v_1)+d(v_1,u)=e(v_1)+1$. Therefore, $e(u)\geq e(v_1)+1$. If 	
$\lvert V_{P(u,y)}\cap\{v_1,v_2\}\rvert =1$  then the
other 
inequality, $e(u)\leq e(v_1)+1$ follows from (\ref{leq ineq A}) and 
(\ref{leq 
ineq C1}).  If  $v_1,v_2 \in P(y,u)$ then 
$e(u)=d(y,u)=d(y,v_1)+d(v_1,u)\leq e(v_1)+1$. \\

\noindent $(ii)$ We claim that there exists $x\in 
V_G\setminus V_B$ such that 
$e(v)=d(x,v)$. If $e(v)\geq 3$ then the claim follows. Consider the case $e(v)= 
2$. Suppose that the claim does not hold. Then $d(x,v)<e(v)$ for all $x \in 
V_G\setminus V_B$. That is, $v$ is adjacent to  all $x \in V_G\setminus V_B$. 
This leads to $\diam(G)\leq 3$, a contradiction to the hypothesis. Hence the 
claim. Let $u,v \in 
V_1(B)$. By Remark \ref{p6}, any path $P(u,x)$ passes through $v$. So,
$e(u)\geq d(x,u)=d(x,v)+d(v,u)=e(v)+2$. This implies $e(u)\geq 4$ and hence  
$e(u)=d(u,y)$ for some $y\in V_G\setminus V_B$.
Note that $e(u)= d(y,u)=d(y,v)+2\leq e(v)+2$. Hence  $e(u)=e(v)+2$. Similarly, 
the 
proof follows for the remaining cases.
\end{proof}

Let $\diam(G) \geq 4$ and let $B$ be a block of $G$ such 
that $u\in V_B\setminus C_G$. Suppose that $B$ is a leaf block and  $v\in 
V_B\cap C_G$. From Lemma 
\ref{eccentricity lemma}, we have $e(u)>e(v)$. If $B$ is a bridge block such that the cut-vertices 
$v_1$ and $v_2$ of $B$ are not in the same partite set then also we have 
$e(u)>e(v_i)$ for some $i$. 
Suppose that  the cut-vertices 
$v_1$ and $v_2$ of the bridge block $B$ are  in the same partite set, say 
$V_1(B)$, then $e(u)>e(u_0)$ when $u\in V_1(B)$ and  $u_0\in V_2(B)$. Hence 
$u \not \in C(G)$ in the above cases. Therefore, a non-cut-vertex $u$ of the block $B$ can belong 
to $C(G)$ only when $B$ is a bridge block whose cut-vertices are in one 
partite set of $B$ and  $u$ lies in the other partite set of $B$. Precisely, 
the  result is given in the following remark. 
\begin{remark}\label{non_cutvertex_center}
Let $G\in \mathscr{B}$ with $\diam(G) \geq 4$ and 
let  $x \in V_1(B)$ for some block $B$ of $G$. If $x \in C(G) \setminus 
C_G$, then 
$B$ is a bridge block with $\lvert V_2(B)\cap C_G\rvert=2$.
\end{remark}
The next lemma is a main tool in finding the central vertices of the graphs in 
$\mathscr{B}$.
\begin{lemma}\label{center_rmk}
Let $G\in \mathscr{B} $ and let $\diam(G)\geq 4$. Then $C(T_G)\subseteq C(G)$.
\end{lemma}

\begin{proof}
Let $z\in C(T_G)$. Then by Lemma \ref{ecc_w.r.t T_G}, $e_{G}(z)= e_{T_G}(z)$. 
Let $u\in V_G$.
Suppose that  $u\in V_{T_G}$. Then 
$e_G(u)=e_{T_G}(u) \geq e_{T_G}(z)=e_{G}(z) $.
If $u \not \in V_{T_G}$, then $u \not \in C_G$ because $C_G \subseteq V_{T_G}$.
Assume 
that $u \in V_B$ for some block $B$.
If $B$ is a leaf block such that $v\in V_B\cap C_G$, then by Lemma 
\ref{eccentricity lemma}, we have
$e_G(u)\geq e_G(v)=e_{T_G}(v) \geq e_{T_G}(z).$
Suppose that $B$ is a bridge block such that $v_1,v_2\in V_B\cap C_G$. Consider 
the case $v_1,v_2\in V_1(B)$ and 
$u\in V_2(B)$. Then, by Lemma \ref{eccentricity lemma}, $e_G(x)=e_G(u)$ 
for all $x\in V_2(B)$. In particular, by the 
construction of $T_G$, there exists $u_0\in V_2(B)\cap V_{T_G}$ such that 
$e_G(u_0)=e_G(u)$. This implies 
$e_G(u)=e_G(u_0)= e_{T_G}(u_0)\geq e_{T_G}(z).$
For the remaining cases of $B$, it directly follows from Lemma 
\ref{eccentricity lemma} that 
$e_G(u)\geq e_{T_G}(z)= e_{G}(z)$.
Hence 
$z\in C(G)$, which completes the proof.
\end{proof}

Recall the following result to  study some properties of the central vertices 
of $T_G$.

\begin{lemma}[\cite{BaRa2012}, see \cite{MaKa2022,LiWB2022}]\label{center of  a tree}
Let $T$ be a tree on $n\geq3$ vertices and $\diam(T)=m$. Then the center 
$C(T)$ of $T$ has 	either a single vertex $z_1$ or two adjacent vertices 
$z_1$ and $z_2$ such that 	$e(z_i)= \lceil {\frac{m}{2}}  \rceil$ for all $i=1,2$, 
where $\lceil 
{\frac{m}{2}}\rceil=\frac{m}{2}$ if 
$m$ is an even integer, and $\lceil {\frac{m}{2}}\rceil=\frac{m+1}{2}$ if $m$ is 
an  odd integer. 
Moreover, $\lvert C(T)\rvert=1$ if $\diam(T)$ 
is 
even, and $\lvert C(T)\rvert=2$ if $\diam(T)$ is odd.
\end{lemma}

Let  $G\in \mathscr{B}$ be such that $\diam(G)$ is odd and greater than three. Then by Lemmas 
\ref{diam of G and T_G} and 
\ref{center of  a tree}, $\lvert C(T_G)\rvert=2$. Assume that $C(T_G)=\{z_1,z_2\}$. Again 
by Lemma 
\ref{center of  a tree}, the vertices $z_1$ and $z_2$ are adjacent, and by  
Lemma \ref{center_rmk},  $z_1,z_2 \in C(G)$. Suppose that 
both  $z_1$ and $z_2$ are not in $C_G$. Then by Remark 
\ref{non_cutvertex_center},	
the block $B$ of $G$ containing $z_1$ and $z_2$ should be a bridge block such 
that the 
cut-vertices of $G$ in $B$ are in one partite set (say $V_1(B)$), and $z_1$ 
and $z_2$ are in the other partite set ($V_2(B)$), which is not possible 
because $z_1$ and $z_2$ are adjacent. Therefore, we arrive at the following 
remark.

\begin{remark}\label{center of T_G has a cutvertex}
Let  $G\in \mathscr{B}$ with  $\diam(G)\geq 4$. 
Then
$\lvert C(T_G)\cap C_G\rvert\geq 1$ whenever $\lvert C(T_G)\rvert=2$.
\end{remark}
The following result presents the collection of all  central vertices of graphs 
in $\mathscr{B}$. This will be used frequently in the proofs of the next 
section.

\begin{theorem}\label{center of G lemma}
Let $G\in \mathscr{B}$ be such that $\diam(G) \geq 4$ and let $T_G$ be its 
associated 
tree.
\begin{itemize}
\item [(i)] Suppose that $\diam(G)$ is even and $C(T_G)=\{z\}$ where $z\in 
V_1(B)$ 
for some block $B$. Then 
\begin{eqnarray*}
C(G)= \begin{cases}  
	\{z\} & \text{if } z\in C_G,\\
	V_1(B)& \text{if } z\not\in C_G.
\end{cases}
\end{eqnarray*}
\item [(ii)] Let $\diam(G)$ be odd and $C(T_G)=\{z_1,z_2\}$ where $z_1\in 
V_1(B)$ 
and $z_2\in V_2(B)$ for some block $B$. Then 
\begin{eqnarray*}
C(G)= \begin{cases}  
	\{z_1,z_2\} & \text{if } z_1,z_2\in C_G,\\
	\{z_1\}\cup V_2(B)& \text{if } z_1\in C_G ~\text{and}~ z_2\not\in C_G,\\
	\{z_2\}\cup V_1(B)& \text{if } z_1\not\in C_G ~\text{and}~ z_2\in C_G.
\end{cases}
\end{eqnarray*}
\end{itemize}
\end{theorem}
\vspace{-.5cm}
\begin{proof}
$(i)$
\textbf{Case 1}:
Let $z\in C_G$. By Lemma \ref{center_rmk}, 
it is 
sufficient to show that $C(G)\subseteq\{z\}$. Let $x \in C(G)$. We first 
claim 
that  $x \in C_G$. On the contrary, assume that  $x \not\in C_G$. Then by 
Remark \ref{non_cutvertex_center} and Lemma \ref{eccentricity lemma}, $x\in 
B^{\prime}$ for some bridge block 
$B^{\prime}$ such that the cut-vertices of $G$ in $B^{\prime}$ belong to 
$V_1(B^{\prime})$ and $x 
\in V_2(B^{\prime})$ with  $e_G(x)=e_G(y)$ 
for all $y\in 
V_2(B)$. In particular, there exists $y_0\in V_2(B^{\prime})\cap V_{T_G}$ 
such that 
$e_G(x)=e_G(y_0)$ where the existence of $y_0$ is guaranteed by the 
construction 
of $T_G$. Note that $y_0\not \in C_G$ as $\lvert C_G\cap V_1(B^{\prime})\rvert=2$ and 
$y_0\in  V_2(B^{\prime})$.
By Lemma \ref{ecc_w.r.t T_G}, $e_G(y_0)=e_{T_G}(y_0)$, and we 
have 
$e_{T_G}(y_0)=e_G(x)\leq e_{G}(z)=e_{T_G}(z)$. Since $C(T_G)=\{z\}$, 
$e_{T_G}(y_0)\geq 
e_{T_G}(z)$ which implies	$y_0\in C(T_G)$.
That is, $y_0=z$, which is not possible because  $z\in C_G$ and $y_0 \not \in 
C_G$. Hence the claim 
$x \in C_G$ follows. Therefore, $x\in V_{T_G}$ and $e_G(x)=e_{T_G}(x)$. We have 
$e_{T_G}(x)\geq e_{T_G}(z)= e_{G}(z)\geq e_{G}(x)$,
which yields $e_{T_G}(x)=e_{T_G}(z)$. That is, $x\in C(T_G)=\{z\}$ and 
hence $x=z$. Therefore, $C(G)=\{z\}$.\\	
\textbf{Case 2}: Assume that $z\not \in C_G$. Since $z\in V_1(B)$, it follows
from Remark \ref{non_cutvertex_center} that $B$ must be a bridge block whose 
cut-vertices are in $V_2(B)$. Also,  by Lemma \ref{eccentricity lemma}, we get
$e_G(z)=e_G(u)$ for all $u\in V_1(B)$. Using the fact that $z\in 
C(T_G)\subseteq 
C(G)$, we  have $V_1(B)\subseteq C(G)$. To prove 
$C(G)\subseteq V_1(B)$, let $x \in C(G)$. We now show that  $x \not \in C_G$. 
Suppose that  $x  \in C_G$. Then  $x\in V_{T_G}$ and by Lemma \ref{ecc_w.r.t 
T_G}, we have
$e_{T_G}(x)=e_G(x)\leq e_G(z)=e_{T_G}(z)\leq e_{T_G}(x)$. This implies
$e_{T_G}(z)= e_{T_G}(x)$ and hence
$x\in 
C(T_G)=\{z\}$. That is, $x=z$, which is absurd as $x\in C_G$ and $z\not 
\in C_G$. Therefore,  $x \not \in C_G$.
%
Since $x\in C(G)$, by Remark \ref{non_cutvertex_center}, 
$x  \in V_2(B^{\prime})$ for some bridge block 
$B^{\prime}$ with $\lvert V_1(B^{\prime})\cap C_G\rvert=2$. As in the previous case, we 
have $e_G(x)=e_G(y_0)$  for some 
$y_0\in   V_2(B^{\prime})\cap V_{T_G}$ and $y_0=z$. This yields $y_0\in 
V_1(B)$. Since $y_0\in V_B\cap 
V_{B^\prime}$ and $y_0\not \in C_G$, it follows by $(P_3)$ that  
$B=B^{\prime}$ and 
$V_2(B^{\prime})=V_1(B)$. 
This implies $x\in V_1(B)$, which completes the proof in   this case.\\  

\noindent
$(ii)$ It is similar to the proof of item $(i)$.
\end{proof}
\begin{remark}
It is clear from Theorem \ref{center of G lemma} that all the
central vertices of a graph $G\in \mathscr{B}$ with $\diam(G)\geq 4$  
lie completely in one specific block $B$ of $G$. In particular, all the 
central vertices of $G$ belong to exactly one 
partite set of $B$ if and only if  $\diam(G)$ is even. 
\end{remark}

\section{Inertia, spectral symmetry and irreducibility }\label{inertia and spectral symmetry}

The  inertias of the 
distance matrices of  trees, and the eccentricity matrices of 
lollipop graphs,  trees,  clique trees and coalescence of certain 
graphs 
are computed in the literature, see 
\cite{Bapa2010,LiWB2022,MaKa2022,MGKA2020,PaSP2021}.  Along 
these lines, here we find the inertias of the eccentricity matrices 
of  graphs in $\mathscr{B}$.

It is shown that
the eigenvalues of the adjacency matrix $A(G)$ of a graph $G$ are 
symmetric about the origin   if and only if $G$ 
is 	bipartite \cite{Bapa2010}. In \cite{MaKa2022}, 
it is proved  that the eigenvalues of 
the eccentricity matrix $E(T)$ of a tree $T$  are symmetric 
about the origin if and only if  $\diam(T)$ is odd. 
A similar equivalence is established 
for a 
subclass of 
block graphs (clique trees) in  \cite{LiWB2022}. Motivated by these, in this 
section, we prove an analogous result for  the class 
$\mathscr{B}$.
In the last part of this section, we prove the irreducibility of  the 
eccentricity matrices of  graphs in $\mathscr{B}$.

\vspace{-.5cm}
\subsection{ Inertia and spectral symmetry  of eccentricity matrices 
of graphs in $\mathscr{B}$ with odd diameters}\label{odd_diam}

In this subsection, we consider  graphs in $\mathscr{B}$ with odd diameters and 
find the inertias of the eccentricity matrices of these graphs. Also, we show the 
spectral symmetry (with respect to the origin) of the above considered matrices.
We begin with an example.
\begin{example}
Consider the bi-block graph $G\in \mathscr{B}$ and the associated tree $T_G$ given in Example 
\ref{Illustrating_eg}. For the purpose of illustrating Theorem \ref{odd_thm}, let us relabel the vertices $v_7$, $v_{9}$ and  $v_{10}$ by $z_1$, $z^{\prime}$ and $z_2$, respectively. Note that $\diam(G)=\diam(T_G)=7$, and 
$C(T_G)=\{z_1,z_2\}$ where $z_1$ and $z_2$ are adjacent with $z_1 \in C_G$ and 
$z_2 \not \in C_G$. Let $B$ be the block of $G$ containing the edge $z_1z_2$ such 
that 
$z_1\in V_1(B)$ and $z_2\in V_2(B)$.
It is clear that	$B$ is a bridge block with  $z^{\prime}\in C_G\cap V_1(B)$ and $z^{\prime}\neq z_1$. 	Let
$G^{\prime}$ be 	the subgraph of $G$ obtained by deleting all the 
edges of the block $B$ in $G$. Let $C_1$ and $C_2$ be the  components of 
$G^{\prime}$  
containing the vertices $z_1$ and $z^{\prime}$, respectively.
Clearly, the 
sets 
\begin{align*}
U_1 &= \{ x\in V_{C_1}: d(x,z_1)=3\}
=\{v_1,v_2,v_6\}, \\
U_{2} &= \{ y\in V_{C_1}:0 \leq d(y,z_1)< 3\}
=\{v_3,v_4,v_5,z_1\}, \\
U_3 &= \{ x\in V_{C_2}: d(x,z^{\prime})=2\}
=\{v_{15},v_{16}\},\\
U_{4} &= \{ y\in V_{C_2}:0 \leq d(y,z^{\prime})<2\}
=\{z^{\prime},v_{13},v_{14}\},\\
U_{5} &= V_1(B)\setminus\{z_1,z^{\prime}\}
=\{v_8\}, \quad \text{and}\\
U_{6} &= V_2(B)
=\{z_2,v_{11},v_{12}\}~\text{	partitions $V_{G}$.}~
\end{align*}
%
Then  the 
eccentricity 
matrix 	 of $G$ can be written in the following   form:
\begin{equation*}
	E(G)=	\begin{blockarray}{ccccccc}
		& U_1 & U_2 & U_3 & U_4&U_5&U_6 \\
		\begin{block}{c(cccccc)}
			U_1 & O&O & 7J  & P & 5J&4J\\
			U_2 & O&O& Q & O&O&O\\
			U_3 &7J^{\prime} & Q^{\prime} & O & O&O&O\\
			U_4 & P^{\prime} & O& O & O&O&O\\
			U_5 & 5J^{\prime} & O& O & O&O&O\\
			U_6 & 4J^{\prime} & O& O & O&O&O\\
		\end{block}
	\end{blockarray},
\end{equation*}
where
	\begin{equation*}
				P=
			\begin{blockarray}{cccc}
					& v_{13} & v_{14} &  z^{\prime} \\
					\begin{block}{c(ccc)}
							v_1 &6&6&5\\
							v_2 &6&6&5\\
							v_6 &6&6&5\\
						\end{block}
				\end{blockarray},\quad\text{and}\quad
			Q=
			\begin{blockarray}{ccc}
					& v_{15}  & v_{16} \\
					\begin{block}{c(cc)}
							v_3 &6&6\\
							v_4 &6&6\\
							v_5 &5&5\\
							z_1 &4&4\\
						\end{block}
				\end{blockarray}.
		\end{equation*}
Using SAGEMATH, it is computed 
that the eigenvalues 
of $E(G)$ are $0$,  $30.0375$, $11.3025$, $-11.3025$ and $30.0375$ with 
respective multiplicities $12,1,1,1$ and $1$. 
Thus,   the spectrum of $E(G)$ is symmetric with 
respect 
to the origin, and $\In(E(G))=(2,2,12)$.
\end{example}
\begin{theorem}\label{odd_thm}
Let $G \in \mathscr{B}$ be such that $\diam(G)=2m+1$, $m\geq 2$. Then the 
following statements hold:
\begin{itemize}
	\item [(i)] $\In(E(G))=(2,2,n-4)$, where $n$ is the number of vertices 
	of $G$.
	\item [(ii)] The spectrum of $E(G)$ is symmetric with respect to the 
	origin.
\end{itemize}
\end{theorem}
\vspace{-.5cm}
\begin{proof}
Consider the tree $T_G$ associated with $G$.
By Lemmas \ref{diam of G and T_G} and  \ref{center of  a tree}, 
$\diam(T_G)=2m+1$ and $\lvert C(T_G)\rvert=2$. Let $C(T_G)=\{z_1,z_2\}$. Again by 
Lemma \ref{center of  a tree}, $z_1$ 
and $z_2$ are adjacent, and $e_{T_G}(z_i)=m+1$ for $i=1,2$. Also, by Lemmas 
\ref{ecc_w.r.t T_G} and
\ref{center_rmk}, $z_1,z_2 \in 
C(G)$ and $e_{G}(z_i)=m+1$ for $i=1,2$.
Let $z_1\in V_1(B)$ and $z_2\in V_2(B)$ for some block $B$ of $G$. By 
Remark \ref{center of T_G has a cutvertex}, $z_i \in C_G$ for at least one 
$i$. 
Without loss of generality, assume that $z_1 \in C_G$.\\

	\noindent$(i)$ \textbf{Case 1}: Suppose that $z_2\not\in C_G$. Since 
	$z_2\in \left(C(G)\cap V_2(B)\right)\setminus C_G$, by 
	Remark 
	\ref{non_cutvertex_center}, $B$ is a
	bridge block whose cut-vertices are in $V_1(B)$. Let $z^{\prime}\neq 
	z_1$ be the
	another cut-vertex of $G$ in $B$. By  $(P_1)$, we have  $\lvert V_2(B)\rvert\geq 
	2$ as  $\{z_1,z^{\prime}\}\subseteq V_1(B)$.
	Now obtain
	a subgraph $G^{\prime}$ from $G$ by deleting all the 
	edges of
	$B$. Clearly,  $V_G=V_{G^{\prime}}$. By Remark \ref{p6} and $(P_3)$,  
	$G^{\prime}$ contains at least 
	four components as $\lvert V_i(B)\rvert\geq 2$ for all $i=1,2$, and the vertices 
	$z_1$ and $z^{\prime}$ lie in different components. Let $C_1$ and 
	$C_2$ be the non-trivial components of 
	$G^{\prime}$  
	containing the vertices $z_1$ and $z^{\prime}$, respectively. 
	Let $x\in V_{C_1}$ and $y\in V_{C_2}$. Then by $(P_3)$,
	\begin{equation}\label{*}
		z_1,z^{\prime} \in \bar{P}(x,y)~\text{for all paths $\bar{P}(x,y)$ 
			between 
			$x$ and 
			$y$ in $G$}.
	\end{equation}
	We prove the following inequalities to 
	find a partition for $V_G$: 
	\begin{equation}\label{**}
		d(a,z_1)\leq m~ \text{for all $a\in V_{C_1}$}\quad\text{and}\quad
		d(b,z^{\prime})\leq m-1~ \text{for all $b\in V_{C_2}$}.
	\end{equation}
	Since $e(z_1)=m+1$, we have $d(a,z_1)\leq m+1$. Suppose $d(a,z_1)= m+1$ then 
	$d(a,z_2)= m+2$ which is not possible because $e(z_2)=m+1$. Hence 
	$d(a,z_1)\leq m$
	which proves the first inequality in 
	$(\ref{**})$. Suppose that the 
	second inequality in 
	$(\ref{**})$ does not hold. That is, there exists  $b\in V_{C_2}$ such that 
	$d(b,z^{\prime})> m-1$. Since $z_1 \in V_{C_1}$ and $b\in V_{C_2}$, using 
	(\ref{*}), we write 	$d(b,z_1)=d(b,z^{\prime})+d(z^{\prime},z_1)
	>(m-1)+d(z^{\prime},z_1)=m+1$. This implies 	$e(z_1)>m+1$, which is 
	a contradiction. So, the second inequality in 
	$(\ref{**})$  holds.	
	
	We now show that the non-trivial components of $G^{\prime}$ are precisely $C_1$ and $C_2$.
	Suppose that $x\in V_G$ and $x\not \in V_{C_1}\cup V_B$.  Let $P(x,z_1)$ be a shortest path between $x$ and $z_1$ in $G$. 
	If $z^{\prime}\not \in P(x,z_1)$ then by Remark \ref{p6}, $V_{P(x,z_1)}\cup V_B=\{z_1\}$. Therefore, $P(x,z_1)$ is a path in  $G^{\prime}$. Since $C_1$ is a component of $G^{\prime}$ containing $z_1$, we have $x\in V_{C_1}$ which is a contradiction. 
	Hence   $z^{\prime} \in P(x,z_1)$. Note that the subpath $P(x,z^{\prime})$ obtained from the shortest path $ P(x,z_1)$ does not contain any vertex other than $z^{\prime}$ from the block $B$.  This implies that $P(x,z^{\prime})$ is a path in  $G^{\prime}$. Since $z^{\prime}\in V_{C_2}$, we have $x\in V_{C_2}$. Thus,  the components of $G^{\prime}$ other than $C_1$ and $C_2$ are  simply complete 
	graphs of order one which 
	arise from $V_B\setminus\{z_1,z^{\prime}\}$.	We now define the following sets to obtain 
	a 
	partition for $V_{G^{\prime}}$:
	\begin{align*}
		U_1 &= \{ x\in V_{C_1}: 
		d(x,z_1)=m\},\quad\quad\quad\quad\quad~~~~
		U_{2} = \{ y\in V_{C_1}:0 \leq d(y,z_1)< m\},\\
		U_3 &= \{ x\in V_{C_2}: 
		d(x,z^{\prime})=m-1\},\quad\quad\quad\quad~~
		U_{4} = \{ y\in V_{C_2}:0 \leq d(y,z^{\prime})<m-1\},\\
		U_{5} &= V_1(B)\setminus\{z_1,z^{\prime}\}~ \text{whenever 
			$\lvert V_1(B)\rvert\geq 3$}, 
		~ \text{and}~~
		U_{6} = V_2(B).
	\end{align*}
	Clearly, $V_{G^{\prime}}=\cup_{l=1}^{6}U_l$ and $U_l\cap U_k=\emptyset$ 
	for $l\neq k$. Also, $z_1\in U_2$, $z^{\prime}\in U_{4}$ and $z_2 \in 
	U_{6}$. 	To see 
	$U_1\neq \emptyset$ and $U_3\neq \emptyset$,  consider $e(z_1)$ and 
	$e(z^{\prime})$. Let 	$e(z_1)=d(b,z_1)$ for some $b\in V_G$. Since 	
	$e(z_1)=m+1 \geq 3$, we get $b \not \in V_B$. Also, by (\ref{**}),  $b \not 
	\in V_{C_1}$. Therefore $b  \in V_{C_2}$, and by (\ref{*}), 
	$d(b,z_1)=d(b,z^{\prime})+d(z^{\prime},z_1)=d(b,z^{\prime})+2$. This implies
	$d(b,z^{\prime})=m-1$ and hence $b\in U_3$. Since $z_2 \in V_B\setminus 
	C_G$, by item $(i)$ of Lemma \ref{eccentricity lemma}, we have $e(z_2)= 
	\max \{e(z_1),e(z^{\prime})\}-1=\max \{m+1,e(z^{\prime})\}-1$. Since 
	$e(z_2)=m+1$, it follows that $e(z^{\prime})=m+2$. Let 
	$e(z^{\prime})=d(a,z^{\prime})$. 
	Then by (\ref{**}), $a \in V_{C_1}$, and by (\ref{*}), we write 
	$d(a,z^{\prime})=d(a,z_1)+d(z_1,z^{\prime})=d(a,z_1)+2$. This gives 
	$d(a,z_1)=m$ and hence $a \in U_1$. 
	
	\noindent	\textbf{Subcase 1.1}:		Assume that $\lvert V_1(B)\rvert\geq 3$. Then $U_5\neq 
	\emptyset$. 
	Thus 
	$\{U_1,U_2,\ldots U_{6}\}$ partitions $V_{G^{\prime}}$ and hence $V_G$ 
	as 
	well. 	To find $E(G)$ explicitly, we now compute  the eccentricity of each 
	vertex 
	in 	$G$.

	Let  $a\in U_1 \cup U_2$. To find $e(a)$, consider 
	$d(a,b_0)$ 
	where $b_0 \in U_3$ is fixed. By (\ref{*}),
	\begin{equation}\label{ecc in U_2}
		e(a)\geq d(a,b_0)=d(a,z_1)+d(z_1,z^{\prime})+d(z^{\prime},b_0)
		=d(a,z_1)+(m+1).
	\end{equation}
	We claim that $e(a)=d(a,z_1)+(m+1)$. Let $e(a)=d(a,x)$ for some $x\in V_G$. 
	By (\ref{ecc in U_2}), we have $e(a)\geq m+1$. If $x  \in V_{C_1}$ then  by 
	(\ref{**}),
	$e(a)=d(a,x)\leq d(a,z_1)+d(z_1,x)\leq d(a,z_1)+m<d(a,z_1)+(m+1)$, a 
	contradiction to (\ref{ecc in U_2}). 
	So, $x \not \in V_{C_1}$. Similarly, we see that
	$x \not \in B$. 
	Therefore, $x\in V_{C_2}$. By (\ref{**}),
	$d(x,z^{\prime})\leq m-1$. This implies 
	$d(a,x)\leq d(a,z_1)+d(z_1,z^{\prime})+d(z^{\prime},x)
	\leq d(a,z_1)+2+(m-1)= d(a,z_1)+(m+1) $ and hence $e(a)\leq 
	d(a,z_1)+(m+1)$. Therefore, $e(a)= d(a,z_1)+(m+1)$ for all $a\in U_1 \cup 
	U_2$. In particular, if $a\in U_1$ then $e(a)=2m+1$ and  if $a\in U_2$ then 
	$e(a)<2m+1$. 
	Similarly, it can be shown that $e(b)= d(b,z^{\prime})+(m+2)$ for all $b\in 
	U_3 \cup U_4$.
	Since $e(z^{\prime})=m+2$ and $e(z_1)=m+1$, by Lemma \ref{eccentricity lemma}, we have 
	$e(a)=m+2$ for all $a\in U_5$ and $e(a)=m+1$ for all  $a\in U_6$.

We now compute the entries of $E(G)$. Let $a \in U_1$ and $b\in V_G$. Then 
$e(a)\geq e(b)=\min\{e(a),e(b)\}$. If $b \in U_1 \cup U_2$ then
$	d(a,b)\leq d(a,z_1)+d(z_1,b)= m+d(z_1,b)<e(b)$,
and hence $E(G)_{a,b}=0$. 
Suppose that  $b \in U_3 \cup U_4$. Using (\ref{*}), we write
$d(a,b)=d(a,z_1)+d(z_1,z^{\prime})+d(z^{\prime},b)
=m+2+d(z^{\prime},b)= e(b)$,
and so $E(G)_{a,b}=2m+1$ if $b\in U_3$ and $E(G)_{a,b}=e(b)$ if $b \in U_4$.  
If $b \in U_5 \cup U_6$ then by $(P_3)$ and Remark \ref{p6}, $z_1\in 
\bar{P}(a,b)$ for all paths $\bar{P}(a,b)$. Therefore,
$d(a,b)=d(a,z_1)+d(z_1,b)=m+d(z_1,b)= e(b)$. Thus, $E(G)_{a,b}$ is
$m+2$ if $b\in U_5$ and  $m+1$ if $b \in U_6$.  Since $m\geq 2$, we  see  that $E(G)_{a,b}=0$ for all $a,b\in U_6$. It can be shown that $d(x,y)=e(x)=\min\{e(x),e(y)\}$ if $x \in U_2$ 
and $y\in U_3$, and  $d(x,y)<\min\{e(x),e(y)\}$ for all other cases. 
Therefore, the eccentricity matrix 
$E(G)$  can be written as
\begin{equation}\label{odd_case1}
	E(G)=
	\begin{blockarray}{ccccccc}
		& U_1 & U_2 & U_3 & U_4&U_5&U_6 \\
		\begin{block}{c(cccccc)}
			U_1 & O&O & (2m+1)J  & P & (m+2)J&(m+1)J\\
			U_2 & O&O& Q & O&O&O\\
			U_3 &(2m+1)J^{\prime} & Q^{\prime} & O & O&O&O\\
			U_4 & P^{\prime} & O& O & O&O&O\\
			U_5 & (m+2)J^{\prime} & O& O & O&O&O\\
			U_6 & (m+1)J^{\prime} & O& O & O&O&O\\
		\end{block}
	\end{blockarray},
\end{equation}
where $P_{a*}=P_{b*}$ for all $a,b \in U_1$ and $Q_{*u}=Q_{*v}$ for all 
$u,v \in U_3$. In fact, if 	
$U_{3}=\{u_{31},u_{32},\ldots,u_{3\lvert U_{3}\rvert}\}$ and 
$U_{4}=\{u_{41},u_{42},\ldots,u_{4\lvert U_{4}\rvert}\}$ then
$P_{a*}=\left(e(u_{41}), e(u_{42}), \ldots,
e(u_{4\lvert U_{4}\rvert)}\right)$ for all $a \in U_1$ and 	
$Q_{*v}=\left(e(u_{31}), e(u_{32}), \ldots,
e(u_{3\lvert U_{3}\rvert)}\right)^{\prime}$ for all $v \in U_3$.
The 
structures of $P$ 
and $Q$ give $\rank(E(G)\left([U_i\mid V_G]\right))=1$ for all 
$i=1,2,3,4$. Also,
for some 
fixed $x\in	U_{4}$, 
\begin{equation*}
	E(G)_{a*}= \frac{1}{e(x)}
	\begin{cases}  
		(m+2) E(G)_{x*} &\text{for all $ a \in U_{5}$},\\
		(m+1)  	E(G)_{x*} &\text{for all $ a \in U_{6}$}.
	\end{cases}
\end{equation*} 
Hence $\rank(E(G)) \leq 4$. Fix $x_1\in U_1$, $x_2\in U_3$. Let 
$P(x_1,z_1)$ and $P(x_2,z^{\prime})$ be shortest paths. 
Choose $y_1\in 
P(x_1,z_1)$ and $y_2\in P(x_2,z^{\prime})$ such that 
$x_i$ is adjacent to $y_i$ for $i=1,2$. To see $\rank(E(G))\geq 4$, 
consider the principal submatrix $R$ of $E(G)$, which is given by
\begin{equation*}
	R=	\bordermatrix{
		&x_1        & x_2   &y_1       &y_2\cr
		x_1   &0 & 2m+1 & 0 & 2m\cr
		x_2    &	2m+1 & 0 & 2m & 0 \cr
		y_1    &0 & 2m & 0 & 0 \cr
		y_2    & 2m & 0 & 0  &0\cr
	}.
\end{equation*}
Note that $\det(R)\neq 0$. This implies $\rank(E(G))\geq 4$ and hence 
$\rank(E(G))= 4$. Applying Theorem \ref{Schurcomplement_det_formula} to  
$R$, we get
$\In(R)=(2,2,0)$, where we have taken 	$A=\left(\begin{smallmatrix} 
	0 & 2m+1\\
	2m+1 & 0
\end{smallmatrix}\right)$. Therefore, by Theorem 
\ref{Interlacing 
	theorem},  $n_-(E(G))\geq2$ and $n_+(E(G))\geq2$. Since 
$\rank(E(G))= 4$, we 
have $\In(E(G))=(2,2,n-4)$.

\noindent		\textbf{Subcase 1.2}:
Suppose that $\lvert V_1(B)\rvert=2$. Then $U_5= \emptyset$. Now, it is 
easy 
to see that $\{U_1,U_2,U_3,U_4,U_6\}$  partitions $V_G$, and 
$E(G)$ is  the principal submatrix 
of the matrix given in (\ref{odd_case1}), obtained 
by deleting the rows 
and columns corresponding to 
the vertices in $U_5$. Hereafter, the proof is the same as in subcase 1.1.

\noindent		\textbf{Case 2}: Assume that $z_2\in C_G$. That is, $B$ is a
bridge block such that  $z_1\in V_1(B)\cap C_G$ and $z_2\in V_2(B)\cap 
C_G$.
We obtain the non-trivial components $C_1$ and $C_2$  of $G^{\prime}$  
such that $z_1 \in V_{C_1}$ and $z_2\in V_{C_2}$ where
$G^{\prime}$ is the subgraph constructed from $G$ by deleting all the 
edges of
$B$.  As in case 1, it is easily  seen that  $ d(x,z_1)\leq m$ for all 
$x\in V_{C_1}$ and $ d(y,z_2)\leq m$ for all $y\in V_{C_2}$.
Define
\begin{align*}
	U_1 &= \{ x\in V_{C_1}: d(x,z_1)=m\},\quad
	U_{2} = \{ y\in V_{C_1}:0 \leq d(y,z_1)< m\},\\
	U_3 &= \{ x\in V_{C_2}: d(x,z_2)=m\},\quad
	U_{4} = \{ y\in V_{C_2}:0 \leq d(y,z_2)<m\},\\
	U_{5} &= V_1(B)\setminus\{z_1\}~ \text{whenever $\lvert V_1(B)\rvert\geq 2$}, 
	\quad \text{and}\\
	U_{6} &= V_2(B)\setminus\{z_2\}~\text{whenever $\lvert V_2(B)\rvert\geq 2$}.
\end{align*}
\textbf{Subcase 2.1}: Suppose that $\lvert V_1(B)\rvert\geq 
2$. 
Then, by ($P_1$), $\lvert V_2(B)\rvert\geq 2$. We now see that
$\{U_1,U_2,\ldots, U_{6}\}$ partitions $V_G$.   Let  
$a,b \in V_G$. It is 
easy to verify that
\begin{align*}
	e(a)&= \begin{cases}  
		2m+1 & \text{if } a\in U_1\cup U_3,\\
		d(a,z_1)+(m+1) & \text{if } a\in U_2,\\
		d(a,z_2)+(m+1) & \text{if } a\in U_4,\\
		m+2 & \text{if } a\in U_5\cup U_6.
	\end{cases}
\end{align*}
Then $E(G)$ in this case takes the following form:
\begin{equation}\label{odd_case2}
	E(G)=
	\begin{blockarray}{ccccccc}
		& U_1 & U_2 & U_3 & U_4&U_5&U_6 \\
		\begin{block}{c(cccccc)}
			U_1 & O&O & (2m+1)J  & P & (m+2)J&O\\
			U_2 & O&O& Q & O&O&O\\
			U_3 &(2m+1)J^{\prime} & Q^{\prime} & O & O&O&(m+2)J\\
			U_4 & P^{\prime} & O& O & O&O&O\\
			U_5 & (m+2)J^{\prime} & O& O & O&O&O\\
			U_6 & O & O& (m+2)J^{\prime} & O&O&O\\
		\end{block}
	\end{blockarray},
\end{equation}
where all the rows of $P$ are identical and all the columns of $Q$ are 
the same. 
Hereafter, the proof is similar to that of case 1.

Let $\widetilde{E(G)}$ denote the $4 \times 4$ block matrix, which is 
the  leading principal submatrix 
of $E(G)$  in (\ref{odd_case2}).

\noindent	\textbf{Subcase 2.2}: Let $\lvert V_1(B)\rvert=1$. Then by ($P_1$),  $\lvert V_2(B)\rvert= 1$. Note that
$\{U_1,U_2,U_3, U_{4}\}$ partitions $V_G$, and  
$E(G)=\widetilde{E(G)}$. The rest of the proof is similar 
to case $1$.\\

\noindent$(ii)$		 Consider the matrix $E(G)$  given in (\ref{odd_case2}).
Suppose that 
$$E(G)\textbf{v}=\mu\textbf{v} ~\text{for some $0\neq \mu \in \mathbb{R}$ and 
}~
\textbf{v}= (\mathbf{v_{1}}^{\prime},\mathbf{v_{2}}^{\prime},\ldots,\mathbf{v_{6}}^{\prime})^{\prime} \in \mathbb{R}^n,$$
where $\mathbf{v}$ is partitioned according to the partition of $E(G)$. Then it 
is not 
difficult to verify that $E(G)\tilde{\textbf{v}}=-\mu\tilde{\textbf{v}}$
where $\tilde{\mathbf{v}} =(\mathbf{v_{1}}^{\prime},\mathbf{v_{2}}^{\prime},-\mathbf{v_{3}}^{\prime},-\mathbf{v_{4}}^{\prime},-\mathbf{v_{5}}^{\prime},\mathbf{v_{6}}^{\prime})^{\prime}$. 	Also, 
the multiplicities of  $\mu$ and $-\mu$ are equal. 

Similar to the above  case, by employing suitable eigenvectors, 
the result can be verified  for the remaining subcases of item $(i)$.
\end{proof}
For $m=1$, Theorem \ref{odd_thm} need not be true. The following example 
illustrates this.
\begin{example}\label{counter ex}
Consider the following graph $H$ and its eccentricity matrix $E(H)$:
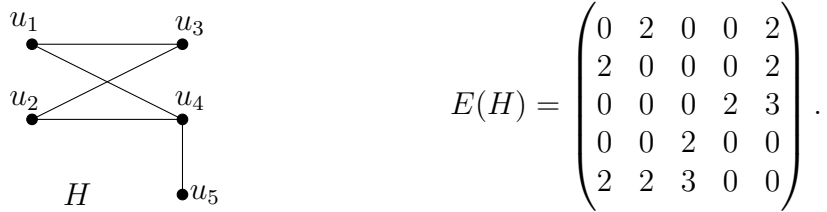
\begin{multicols}{2}
\begin{center}	
	\begin{tikzpicture}[scale=1]
		
		\draw[fill=black] (2,-1) circle (2pt);
		
		\draw[fill=black] (2,1) circle (2pt);
		\draw[fill=black] (2,0) circle (2pt);
		\draw[fill=black] (0,1) circle (2pt);
		\draw[fill=black] (0,0) circle (2pt);
		
		\node at (-.1,.25) {$u_2$};
		\node at (-.1,1.3) {$u_1$};
		\node at (2.1,1.25) {$u_3$};
		\node at (2.1,.25) {$u_4$};
		\node at (2.3,-1) {$u_5$};
		
		\draw[thin] (0,1)--(2,1);
		\draw[thin] (0,1)--(2,0);
		\draw[thin] (2,0)--(2,-1);
		\draw[thin] (0,0)--(2,1);
		\draw[thin] (0,0)--(2,0);
		\node at (.6,-1) {$H$};
	\end{tikzpicture}
	\hspace*{-1.5cm}
	\vfil
	$E(H)=\begin{pmatrix}
		0&2&0&0&2\\
		2&0&0&0&2\\
		0&0&0&2&3\\
		0&0&2&0&0\\
		2&2&3&0&0\\
	\end{pmatrix}.$
	\hspace*{1.5cm}
\end{center}
\end{multicols} 
\noindent  Note that $\diam(H)=3$ and  the rows and columns of $E(H)$ are 
indexed by 
$\{u_1,u_2,u_3,u_4,u_5\}$. The  eigenvalues of 
$E(H)$ are 
$2,-2,-4.1394,-0.7849$ and 
$4.9243$. 
So, 
the spectrum of $E(G)$ is not symmetric about origin.
\end{example}
\subsection{Inertia of eccentricity matrices 
of graphs in $\mathscr{B}$ with even diameters}\label{even_diam}
Let $G\in\mathscr{B}$ be such that $\diam(G)$ is even. In this subsection, we compute the inertia of $E(G)$ and show that the spectrum of $E(G)$ is not symmetric with respect to the origin. As a consequence, we obtain the main result of this section which characterizes the spectral symmetry of $E(G)$ where  $G\in\mathscr{B}$, see Theorem  \ref{symm-tree}.

The following lemma is needed to prove Theorem \ref{even_thm}.
\begin{lemma}\label{diametral path contains a central vertex}
If $G\in\mathscr{B}$ then every diametrical path in $G$ contains a central 
vertex of $G$.
\end{lemma}
\vspace{-.5cm}
\begin{proof}
Let $P(u,v)$ be a diametrical path in $G$ where $u,v \in V_G$. We first 
prove the result for the case $\diam(G)$ is even.  Assume that 
$\diam(G)=2m$ for 
some $m\geq 2$. Then  by Lemmas \ref{diam of G and T_G} and \ref{center of  a tree}, we have $\diam(T_G)=2m$ and $\lvert C(T_G)\rvert=1$. Let $C(T_G)=\{z\}$. 
Since $C(T_G)\subseteq C(G)$, we have $z\in C(G)$. By Lemmas \ref{ecc_w.r.t T_G} and \ref{center of  a tree}, 
$e_G(z)=e_{T_G}(z)=m=e_G(x)$ for all $x \in C(G)$. 
Note that there exists a vertex $a$ in the path  ${P(u,v)}$ such that 
$d(u,a)=m=d(a,v)$. Therefore, $e_G(a)\geq m$. We prove the result by showing 
that
$a \in C(G)$. To do this, we claim that 	$d(a,y)\leq m$ for all $y\in 
V_G$. On the contrary, 	assume that $d(a,y_0)>m$ for some $y_0\in V_G$. 
Then $a$ and $y_0$ do not belong to the same 
block, and  $y_0$ does not lie on the path ${P(u,v)}$. Clearly, either $a$ lies 
in ${P(u,y_0)}$ for some shortest path $P(u,y_0)$ or $a$ lies in
${P(v,y_0)}$ for some shortest path $P(v,y_0)$, otherwise $(P_3)$ fails. 
So,
$d(u,y_0)=d(u,a)+d(a,y_0)>d(u,a)+m=2m$ or $d(v,y_0)>2m$, which is  a contradiction. Hence 
the 
result follows in this case. The proof is similar when $\diam(G)$ is 
odd.
\end{proof}

The notion of diametrically 
distinguished vertex is introduced  in \cite{MaKa2022} for  a
tree with 
even diameter which has exactly one central vertex. Motivated by this, in the following definition, we study this notion for the graph class 
$\mathscr{B}$ where the graphs can have more than one central vertices.
\begin{definition}
Let $G \in \mathscr{B}$ and let $u\in V_G$. Then  $u$ is said to be \textit{diametrically 
	distinguished} if there exists a diametrical path containing  the 
vertex 
$u$ and $u$ is adjacent to $z$
for some $z\in C(G)$.
\end{definition}

\begin{theorem}\label{even_thm}
Let $G \in \mathscr{B}$ be such that $\diam(G)=2m$ with $m\geq 2$. Let the center of the tree $T_G$, assoicated with $G$, be $\{z\}$ and $k$ be the number of elements in the center $C(G)$. Then the 
following hold:
\begin{itemize}
	\item [(i)]  If $z \not \in C_G$ then $\In(E(G))=\begin{cases}
		(3,k+1,n-k-4) & \text{when $m=2$}, \\
		(2,2,n-4) & \text{otherwise},
	\end{cases}$\\ where $n$ is the number of vertices 
	of $G$.
	\item [(ii)]  If $z \in C_G$ then $\In(E(G))=(r,r,n-2r)$ 
	where $r$ is the number of distinct 
	blocks of $G$ having a diametrically distinguished vertex.
	\item [(iii)] The spectrum of $E(G)$ is not symmetric with respect to 
	the origin.
\end{itemize}
\end{theorem}
\vspace{-.5cm}
\begin{proof}
By Lemma 	\ref{center_rmk}, we have $z \in C(G)$, and by Lemma 
\ref{center of  a tree}, $e(z)=m$.\\
	
	\noindent		$(i)$	Suppose that $z\not\in C_G$. Let $z \in V_1(B)$ for 
	some 
	block 
	$B$ of $G$. Then by Remark 
	\ref{non_cutvertex_center} and Theorem \ref{center of G lemma}, $B$ is a
	bridge block such that $w_1,w_2\in V_2(B) \cap C_G$ and $C(G)=V_1(B)$.
	Without loss of generality, assume that $e(w_1)\leq e(w_2)$. By Lemma 
	\ref{eccentricity lemma}, $e(z)=e(w_2)-1$ which implies $e(w_2)=m+1$. 
	Since   $w_1 \not \in C(G)$ and $e(z)\leq e(w_1)\leq e(w_2)$, we get
	$e(w_1)=m+1$.
	Construct a subgraph 
	$G^{\prime}$  from $G$ by deleting all the edges of the block
	$B$. Then by $(P_3)$ and $(P_4)$, $G^{\prime}$  has two  non-trivial components $C_1$ and 
	$C_2$  
	containing the vertices $w_1$ and $w_2$, respectively, and the 
	remaining 
	$\lvert V_B\rvert-2$  components of $G^{\prime}$ are complete graphs of order one.
	We claim that 
	$d(x,w_1)\leq m-1$ for all $x \in V_{C_1}$. Suppose 
	there exists $x_0 \in V_{C_1}$ such that $d(x_0,w_1)> m-1$. Then  by 
	($P_3$) and Remark \ref{p6} that $w_1 \in \bar{P}(x_0,z)$ for  all paths $\bar{P}(x_0,z)$. This 
	implies
	$d(x_0,z)=d(x_0,w_1)+d(w_1,z)>m$ which yields $e(z)>m$, 
	a contradiction. Hence, 	$d(x,w_1)\leq m-1$ for all $x \in V_{C_1}$. 
	Similarly, 	$d(y,w_2)\leq m-1$ for all $y \in V_{C_2}$. 
	For $i=1,2$, define 
	\begin{align*}
		U_i &= \{ x\in V_{C_i}: d(x,w_i)=m-1\},\\ 
		U_{2+i} &= \{ y\in V_{C_i}:0 \leq d(y,w_i)<m-1\},\\
		U_{5} &= V_1(B)~
		\text{and}\\
		U_{6} &= V_2(B)\setminus\{w_1,w_2\}~\text{whenever $\lvert V_2(B)\rvert\geq 
			3$}.
	\end{align*}
	Let $x_0 \in V_G$  such that $e(w_1)=d(x_0,w_1)$. This implies $ x_0 \not 
	\in V_{C_1}$ as $e(w_1)=m+1$. 	 Since $m\geq 2$ and $w_1 \in V_2(B)$, $ x_0 \not 
	\in 
	V_B$. Therefore, $x_0 \in V_{C_2}$. By ($P_3$), $w_2 \in \bar{P}(x_0,w_1)$ for all paths $\bar{P}(x_0,w_1)$. Hence $d(x_0,w_1)=d(x_0,w_2)+d(w_2,w_1)=d(x_0,w_2)+2$ which 
	yields 
	$d(x_0,w_2)=m-1$. Therefore, $x_0 \in U_2$. Similarly, using the fact that 
	$e(w_2)=m+1$, we see that $U_1\neq \emptyset$.
	Assume that $\lvert V_2(B)\rvert \geq 3$.
	Now, it is clear that 
	$\{U_i:1\leq i\leq 6\}$ partitions $V_G$.
	
	Since $C(G)=V_1(B)$, $e(a)=e(z)=m$ for all $a \in U_5$. By Lemma \ref{eccentricity lemma}, $e(b)=m+1$ for all $b \in U_6$. 
	By ($P_3$), we see 
	that $w_1,w_2 \in P_G(x,y)$ for all $x \in V_{C_1}$ and $y \in V_{C_2}$.  Fix $a_0 \in U_1$ and $b_0 \in U_2$. Similar to subcase $1.1$ of item $(i)$ in  Theorem \ref{odd_thm}, we can compute the eccentricity of the remaining vertices of $V_G$, using the shortest paths 
	$P(a_0,w_1)$ in $C_1$ and $P(b_0,w_2)$ in $C_2$, which are given below:
	\begin{align*}
		e(a)&= \begin{cases}  
			2m & \text{if } a\in U_1\cup U_2,\\
			d(a,w_1)+(m+1) & \text{if } a\in U_3,\\
			d(a,w_2)+(m+1) & \text{if } a\in U_4.
		\end{cases}
	\end{align*}
	The eccentricity matrix 
	$E(G)$  is given by
	\begin{equation}\label{even_case2}
		E(G)=
		\begin{blockarray}{ccccccc}
			& U_1 & U_2 & U_3 & U_4&U_5&U_6 \\
			\begin{block}{c(cccccc)}
				U_1 & O& (2m)J &O  & P &mJ&(m+1)J\\
				U_2 &(2m)J^{\prime}&O & Q & O &mJ&(m+1)J\\
				U_3 & O& Q^{\prime} & O&O&O&O\\
				U_4 & P^{\prime} & O& O & O&O&O\\
				U_5 & mJ^{\prime} & mJ^{\prime}& O & O&R&O\\
				U_6 & (m+1)J^{\prime} &(m+1)J^{\prime}& O & O&O&O\\
			\end{block}
		\end{blockarray},
	\end{equation}
	where $P_{a*}=P_{b*}$ for all $a,b\in U_1$, $Q_{u*}=Q_{v*}$ for all 
	$u,v\in 
	U_2$ and $R=\begin{cases}
		2(J_k-I_k) & \text{if $m=2$}, \\
		O & \text{if $m\geq 3$}.
	\end{cases}$  
	The structures of $P$ 
	and $Q$ yield that $\rank(E(G)\left([U_i\mid V_G]\right))=1$ for all 
	$i=1,2,3,4$. 
	
	\noindent		\textbf{Case 1}: Assume that $m\geq 3$. Then the principal submatrix 
	$R=O$. 
	For each $ a \in U_{5}\cup  U_{6}$, observe that
	$E(G)_{a*}=\alpha E(G)_{w_1*}+\beta E(G)_{w_2*}$
	for some real numbers $\alpha$ and $\beta$. Hence $\rank(E(G)) 
	\leq 
	4$. The rest of the proof is similar to that of Theorem \ref{odd_thm} by considering the  principal 		submatrix 	$$\begin{pmatrix}
		0 & 2m & 0 & 2m-1\\
		2m & 0 & 2m-1 & 0 \\
		0 & 2m-1 & 0 & 0 \\
		2m-1 & 0 & 0  &0
	\end{pmatrix}.$$
	If  $\vert V_2(B)\rvert=2$ then $U_6=\emptyset$. Therefore,  $E(G)$ is a $5 \times 5$ 		block matrix. The result in this subcase can be verified similarly.
	
	\noindent		\textbf{Case 2}: Let $m=2$. Then $U_3=\{w_1\}$ and 
	$U_4=\{w_2\}$. Note that the principal submatrix $R$ of 
	$E(G)$ in (\ref{even_case2}) is $2(J_k-I_k)$ where $k=\lvert C(G)\rvert=\lvert U_5\rvert$. 
	Let $C=\begin{pmatrix}
		2J_{k\times \lvert U_1\rvert} &2J_{k\times\lvert U_2\rvert}&\mathbf{0}&\mathbf{0}
	\end{pmatrix}$ and $D=2(J_k-I_k)$. Then 
	$D^{-1}=\frac{1}{2(k-1)}J_k-\frac{1}{2}I_k$. Let $A$ denote the $4 
	\times 4$ block leading principal submatrix of $E(G)$ in 
	(\ref{even_case2}) and let $M=A-C^{\prime}D^{-1}C$. We have
	\begin{equation*}
		M=\begin{pmatrix}
			\mu J_{\lvert U_1\rvert}& (\mu+4)J_{\lvert U_1\rvert\times\lvert U_2\rvert} &\mathbf{0}  & 
			3\mathbf{e} \\
			(\mu+4)J^{\prime}&\mu J_{\lvert U_2\rvert} &  3\mathbf{e} &  
			\mathbf{0} \\
			\mathbf{0}^{\prime}& 3\mathbf{e}^{\prime} & 0&0\\
			3\mathbf{e}^{\prime} & \mathbf{0}^{\prime}& 0 & 0\\
		\end{pmatrix}~\text{where}~\mu=\frac{-2k}{k-1}.
	\end{equation*}
	Consider  the principal 
	submatrix 
	$N=\left(\begin{smallmatrix} 
		\mu & \mu+4 & 0&3 \\ 
		\mu+4 & \mu & 3&0 \\
		0 & 3 & 0 &0&\\
		3&0&0&0
	\end{smallmatrix}\right)$ of $M$. Apply Theorem \ref{Schurcomplement_det_formula} to $N$, by taking $A$ as $2\times 2$ leading principal submatrix of $N$,   we get
	$\In(N)=(2,2,0)$. Therefore, by Theorem \ref{Interlacing theorem}, 
	$n_+(M)\geq 
	2$ and $n_-(M)\geq 2$.\\
	\noindent		\textbf{Subcase 2.1}: Suppose that $\lvert V_2(B)\rvert=2$. Then 
	$U_6=\emptyset$ and so $E(G)$ is a $5\times5$ block matrix.  Since 
	$\rank(M)=4$,  we have $\In(M)=(2,2,n-k-4)$. 
	Using Theorem \ref{Schurcomplement_det_formula} to $E(G)$, we get
	$	
	\In(E(G))=\In(2J_k-2I_k)+\In(M)=(1,k-1,0)+(2,2,n-k-4)
	=(3,k+1,n-k-4)$.\\
	\noindent		\textbf{Subcase 2.2}: If $\lvert V_2(B)\rvert\geq 3$ then  $U_6\neq \emptyset$. By subcase $2.1$, we have 
	$\rank(E(G))\geq k+4$. Since each column in $U_6$ is a linear combination of columns in $U_3$ and $U_4$, we have $\rank(E(G))=k+4$. 
	Therefore, the result follows using 	$\In(X)$ and Theorem \ref{Interlacing 
		theorem} where $X$ is a $5 \times 5$ block leading principal submatrix of 
	$E(G)$.\\

	\noindent $(ii)$ Assume 
	that  $z\in C_G$. Then by Theorem \ref{center of G lemma}, $C(G)=\{z\}$. Let
	$\deg(z)=p$. Then by $(P_4)$, $p\geq 2$. Let $w_1,w_2, \ldots, w_p$ be the vertices in 
	$G$ that are adjacent to $z$. Now obtain 
	the subgraph $G^{\prime}$  from $G$ by deleting all the edges that are incident 
	with $z$. 
	By the construction of $G^{\prime}$, it 
	is clear that $V_G=V_{G^{\prime}}$, and  $G^{\prime}$ has a
	component $K_1$ with the vertex set $\{z\}$.   Let the components of 
	$G^{\prime}$  be 
	$C_1,C_2, \ldots,C_q$. 
	Among these, let  $C_1,C_2, 
	\ldots,C_r$ be the components containing the vertices $x_1,x_2, \ldots, x_r$, 
	respectively, such 
	that 
	$d(x_i,z)=m$ for all $i=1,2,\ldots,r$. 
		Since $e(z)=m$, 
		existence of $C_i$ containing such $x_i$ is guaranteed. We claim that 
		$r \geq 2$. 
		Let $P_G(a,b)$ be a diametrical path in $G$ where $a,b\in V_G$.
		Then, by Lemma \ref{diametral path 
			contains a central vertex}, $z \in P_G(a,b)$, and hence $d(a,z)=m=d(b,z)$. 
		Therefore, $a,b \in \cup_{i=1}^r V_{C_i}$. 
		If
		$a\in V_{C_i}$ and $b\in V_{C_j}$ with $C_i\neq C_j$ then the claim follows. 
		Using $(P_3)$, we observe that any 
		path $P^{\prime}_G(a,b)$ between $a$ and $b$ in $G$ contains the central 
		vertex $z$ as $z \in P_G(a,b)$. This implies $a$ and $b$ do not belong to a single component $C_i$ 
		and hence $r\geq 2$.
		
		We now find 
		the possible components of $G^{\prime}$ containing a diametrically 
		distinguished 
		vertex.
		Let  $S$ be the set of all diametrically distinguished vertices of $G$ and let
		$w\in S$. Then $w$ is adjacent to $z$ and there exists a diametrical path 
		$P_G(a_0,b_0)$ containing $w$. By Lemma 
		\ref{diametral path contains a central vertex}, $z \in P_G(a_0,b_0)$. Then 
		either $d(a_0,w)=m-1$ or $d(b_0,w)=m-1$. Without loss of generality, assume 
		that  $d(a_0,w)=m-1$. That is, the path $P_G(a_0,b_0)$ yields a path $ 
		P_G(a_0,w)$ with  $z\not \in P_G(a_0,w)$, and $d(a_0,z)=m$. This implies 
		$a_0\in V_{C_i}$ for some $i=1,2,\ldots,r$. By the maximality of $C_i$, 
		the path 
		$P_G(a_0,w)$ is completely 
		contained in $C_i$. That is, $w\in V_{C_i}$. Hence diametrically 
		distinguished vertices are necessarily belong to $\cup_{i=1}^rC_i$.

	We claim that there are exactly $r$ distinct blocks of $G$ that contain
	a diametrically distinguished vertex.  For each  $i \in \{1,2,\ldots,r\}$, 
	if we 
	prove $V_{C_i}\cap S \neq 
	\emptyset$ and $V_{C_i}\cap S \subseteq B_i$ for some block $B_i$ in $G$ 
	with   $B_1, B_2, \ldots, B_r$ are pairwise 
	distinct then the claim follows. 
	Let $1\leq i \leq r$. 
	By the definition of $C_i$, there exists $x_i \in V_{C_i}$ 
	such that $d(x_i,z)=m$. Let $P_G(x_i,z)$ be a shortest path between $x_i$ and $z$. Let $w_i$ be the vertex in the path $P_G(x_i,z)$ such 
	that $w_i$ is adjacent to $z$ in $G$. 
	Then we see that  $d(x_i,w_i)=m-1$ which follows from the 
	subpath $P_G(x_i,w_i)$ obtained form  
	$P_G(x_i,z)$. Since $x_i\in V_{C_i}$, and  $z \not \in P_G(x_i,w_i)$, by the 
	maximality of $C_i$, $P_G(x_i,w_i)$ lies in $C_i$. 
	Choose $j \in \{1,2,\ldots,r\}$ such that $C_j \neq C_i$. Then there exist $x_j$ and  $w_j$ in $ V_{C_j}$ such that $d(x_j,w_j)=m-1$. Also, the subpath $P_G(x_j,w_j)$ obtained form  a shortest path $P_G(x_j,z)$, completely lies in $C_j$.
	Since the components $C_i$ and  $C_j$ are disjoint, we 
	have $P_G(x_i,w_i)\cap 
	P_G(x_j,w_j)=\emptyset$.
	Note that, by ($P_3$), $z\in \bar{P}_G(x_i,x_j)$ for all paths $\bar{P}_G(x_i,x_j)$. This 
	implies 
	$d(x_i,x_j)=d(x_i,z)+d(z,x_j)=m+m=\diam(G)$. That is, any shortest path 
	between $x_i$ and $x_j$ in $G$ is a diametrical path.
	In particular, the shortest 
	paths  $P_G(x_i,z)$ and $P_G(x_j,z)$ induce a diametrical path containing 
	$w_i$. This implies $w_i\in S$ and hence $V_{C_i}\cap S \neq \emptyset$. 
	
	We next prove that $V_{C_i}\cap S \subseteq B_i$ for some block $B_i$ in $G$. Let $w 
	\in V_{C_i}\cap S$. Assume that $w_i\in V_{B_i}$ for some block $B_i$ in $G$. 
	Suppose that $\lvert V_{C_i}\cap	S\rvert\geq 2$.  
	Let $w^{\prime}	\in V_{C_i}\cap S$ with $w^{\prime}\neq w$. Then	
	$w^{\prime}\in B_i$, otherwise  ($P_3$) fails. Hence  $V_{C_i}\cap S \subseteq 
	B_i$. Given a block $B$ in $G$ there exists $l\in \{1,2,\ldots,q\}$ such 
	that 
	$V_B\setminus\{z\}\subseteq V_{C_l}$, which is clear from the construction of 
	$G^{\prime}$. Since components of $G^{\prime}$ are disjoint, the existed 
	$l$ for the block $B$ is
	unique. This implies $B_i \neq B_j$ for all $i,j \in \{1,2,\ldots,r\}$ with 
	$i\neq j$. 	
	Thus, there are exactly $r$ distinct blocks of $G$ which have 	a diametrically 	distinguished vertex. 
	
	We next show that $\rank\left(E(G)\right)=2r$ by explicitly finding the matrix $E(G)$.
	By suitably relabelling the vertices 
	$w_1,w_2, \ldots, w_p$,  it is assumed that $w_i\in V_{C_i}$ for all $1\leq i 
	\leq r$.
	For each $i \in \{1,2,\ldots,r\}$, define
	\begin{align*}
		U_i &= \{ x\in V_{C_i}: d(x,z)=m\},\\
		U_{r+i} &= \{ y\in V_{C_i}:0 < d(y,z)<m\},~~~\text{ and}\\
		U_{2r+1} &= V_{G}\setminus\cup_{i=1}^r\left(U_i\cup U_{r+i}\right).
	\end{align*}
	It is clear that $x_i\in U_i$, $w_i\in U_{r+i}$, $z \in u_{2r+1}$ and $U_l \cap U_k=\emptyset$ 
	for $l\neq k$. Thus 
	$\{U_1,U_2,\ldots, U_{2r+1}\}$ partitions  $V_G$.
	Let $a\in V_G$. We claim that $e(a)=d(a,z)+m$. Note that 
	$		d(a,x)\leq d(a,z)+d(z,x)\leq d(a,z)+m$ for all $x\in V_G$, which yields $e(a)\leq d(a,z)+m$. Choose an element $y$ such that $y$ and $a$ are in different components of $G^{\prime}$ and $d(y,z)=m$. Since any path between  $a$ and
	$y$ passes through $z$, we have $	d(a,y)=d(a,z)+d(z,y)=d(a,z)+m$, and hence the claim. This implies $e(a)=2m$ for all $a\in U_i$.
	
	\noindent Next we find the entries of $E(G)$ case-by-case. Let $a,b\in V_G$ and $1\leq 
	i,j\leq 
	r$ with $i\neq j$.
	\begin{itemize}
		\item[$\bullet$] If $a\in U_i$ and $b\in U_j$   then 
		$E(G)_{a,b}=2m$ as
		\begin{equation*}
			d(a,b)= d(a,z)+d(z,b)=m+m=2m= \min\{e(a),e(b)\}.
		\end{equation*}
		\item[$\bullet$] Note that  	$E(G)_{a,b}=0$ whenever $a,b\in U_i$ 	
		because
		\begin{equation*}
			d(a,b)\leq d(a,w_i)+d(w_i,b)=(m-1)+(m-1)< \min\{e(a),e(b)\}.
		\end{equation*}
		\item[$\bullet$] Let $a\in U_i$ and $b\in U_{r+j}\cup U_{2r+1}$. 
		Then $e(a)=2m$ and
		\begin{equation*}
			d(a,b)= d(a,z)+d(z,b)=m+d(z,b)=e(b)= \min\{e(a),e(b)\}.
		\end{equation*}
		In this case,  $E(G)_{a,b}=e(b)$.
	\end{itemize}
	
	Since $E(G)$ is symmetric,  $E(G)_{x,y}=e(x)$ for all  $x\in U_{r+j}\cup U_{2r+1}$ and $y \in U_i$  with $i\neq j$. Similarly the remaining entries of $E(G)$ can be easily computed and are zero.
	Hence the eccentricity matrix $	E(G)$ of $G$ can be written in the block form
	\begin{equation}\label{even_case1}
		\begin{blockarray}{ccccccccc}
			U_1 & U_2 &\hdots & U_r & U_{r+1} & U_{r+2} & \hdots & U_{2r} &
			U_{2r+1}\\
			\begin{block}{(ccccccccc)}
				O & (2m)J & \hdots & (2m)J & O & A_{1,r+2} & \hdots & 
				A_{1,2r} 
				& A_{1,2r+1} \\
				(2m)J & O & \hdots & (2m)J & A_{2,r+1} & O & \hdots & 
				A_{2,2r} & 
				A_{2,2r+1} \\
				\vdots & \vdots & \ddots &\vdots &\vdots &\vdots &\ddots 
				&\vdots & \vdots \\
				(2m)J & (2m)J  & \hdots & O & A_{r,r+1}  & A_{r,r+2} 
				&\hdots & O & 
				A_{r,2r+1} \\
				&&&&&&&&\\
				O  & A_{2,r+1}^{\prime} & \hdots & A_{r,r+1}^{\prime} & O & 
				O & 
				\hdots & O & O\\
				A_{1,r+2}^{\prime} & O & \hdots & A_{r,r+2}^{\prime} & O& 
				O & \hdots 
				& O & O\\
				\vdots & \vdots & \ddots &\vdots &\vdots &\vdots &\ddots 
				&\vdots & \vdots \\
				A_{1,2r}^{\prime} & A_{2,2r}^{\prime} & \hdots & O & O & O 
				& \hdots & 
				O & O\\
				&&&&&&&&\\
				A_{1,2r+1}^{\prime} & A_{2,2r+1}^{\prime} & \hdots & 
				A_{r,2r+1}^{\prime}&O & O & \hdots & O & O\\
			\end{block}
		\end{blockarray}.
	\end{equation}
	Since the matrix obtained in (\ref{even_case1}) is of the form 
	$\varepsilon(T)$
	given in \cite[Theorem $3.2$]{MaKa2022}, the remainder of the proof is similar 
	to that of Theorem $3.2$ in 
	\cite{MaKa2022}.\\
	
	\noindent $(iii)$
	\textbf{Case 1}: Consider the case $\rank(E(G))=2l$ where $l\in \{r,2\}$. Then 
	by Theorem \ref{Co-eff_char_poly}, 
	the characteristic polynomial of 
	$E(G)$ can be written as
	\begin{equation*}
		\chi_x(E(G))={x}^{n-2l}({x}^{2l}-\delta_1{x}^{2l-1}+\delta_2{x}^{2l-2}-
		\cdots+(-1)^{2l-1}\delta_{2l-1}x+(-1)^{2l}\delta_{2l}),
	\end{equation*}
	where $\delta_{2l}\neq 0$ and $\delta_i$ is the sum of all principal minors of order $i$, for all $i=1,2,\ldots, 2l$.
	To prove the spectrum of $E(G)$ is not symmetric with respect to the origin,  
	by Lemma 
	\ref{symmetric_about_origin_equiv_lemma}, it enough to find some 
	$i\in\{1,2,\ldots,2l-1\}$ such that $\delta_{i}\neq 0$ and 
	$\delta_{i+1}\neq 0$.
	It is clear from  (\ref{even_case2}) and (\ref{even_case1}) that every
	$2\times 2$ and $3 
	\times 3$ principal sumatrices $Q_1$ and 
	$Q_2$ of $ 
	E(G)$ are, respectively, in the following form:
	\begin{equation*}
		Q_1=\begin{pmatrix} 
			0 & \alpha \\ 
			\alpha & 0 
		\end{pmatrix} ~\text{and}~
		Q_2= \begin{pmatrix} 
			0 & \alpha & \beta \\ 
			\alpha & 0 & \gamma \\
			\beta & \gamma & 0 
		\end{pmatrix}
		~\text{where $\alpha,\beta$ and $\gamma$ are some non-negative integers.}
	\end{equation*}
	Since $\det(Q_1)=-\alpha^2$ and  $\det(Q_2)=2\alpha\beta\gamma$,
	we have $\delta_2\leq 0$ and $\delta_3\geq 0$. Also, the following principal 
	submatrices have non-zero determinant:
	\begin{equation*}
		\ \bordermatrix{%
			&x_1        & x_2              \cr
			x_1    & 0         &2m            \cr
			x_2    &2m         & 0            \cr
		}, \qquad \text{and} \qquad
		\ \bordermatrix{%
			&x_1        & x_2          &z\cr
			x_1    & 0         &2m            &m\cr
			x_2    &2m        & 0            &m\cr
			z    &m          &m             &0\cr
		},
	\end{equation*}
	where $x_i \in U_i$ for $i=1,2$, and $z \in C(T_G)$.
	Hence  $\delta_2< 0$ and $\delta_3>0$. 
	
	\noindent\textbf{Case 2}: Suppose that $\rank(E(G))=k+4$. Since $k=\lvert V_1(B)\rvert$ 
	with cut-vertices
	$w_1,w_2$ in $V_2(B)$, by ($P_1$), we have $k\geq 2$. If $k=2$ then the proof 
	follows similarly to the previous case. 
	If $k\geq 3$ then by item $(ii)$, 
	$n_-\left(E(G)\right)\geq4>n_+\left(E(G)\right)$ and hence the proof.
\end{proof}

We are now ready to state the main result of this section.
\begin{theorem}\label{symm-tree}
	Let $G \in \mathscr{B}$ with $\diam(G)\geq 4$. Then the eigenvalues of 
	${E}(G)$ are symmetric 
	with respect to the origin if and only if $\diam(G)$ is odd.
\end{theorem}
\vspace{-.5cm}
\begin{proof}
	It follows from Theorems \ref{odd_thm} and \ref{even_thm}.
\end{proof}

The following example shows that Theorem \ref{odd_thm} need not be true 
for a 
general bi-block graph.  
\begin{example}
	Consider  the graph $G \not \in \mathscr{B}$ and the matrix $E(G)$
	which are given below:
	\begin{multicols}{3}
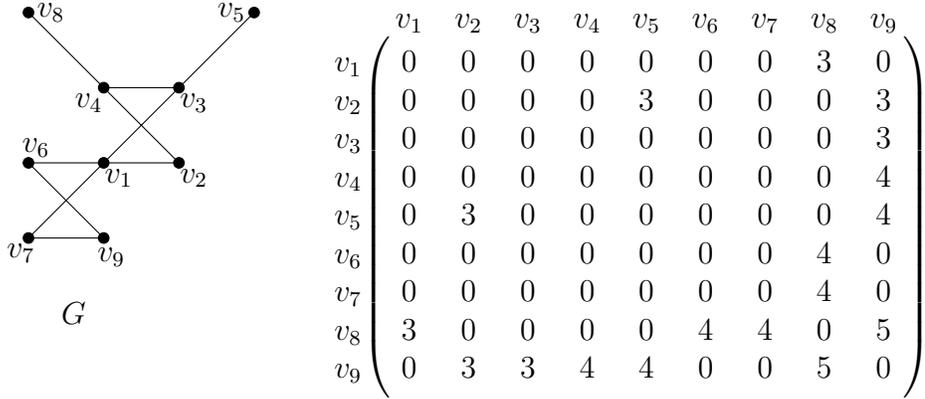
		\begin{center}
			\begin{tikzpicture}[scale=1]
				
				\draw[fill=black] (0,0) circle (2pt);			
				\draw[fill=black] (1,0) circle (2pt);
				\draw[fill=black] (1,1) circle (2pt);
				\draw[fill=black] (0,1) circle (2pt);
				\draw[fill=black] (2,1) circle (2pt);
				\draw[fill=black] (1,2) circle (2pt);
				\draw[fill=black] (2,2) circle (2pt);
				\draw[fill=black] (0,3) circle (2pt);
				\draw[fill=black] (3,3) circle (2pt);

				\draw[thin] (0,0)--(1,0)--(0,1)--(1,1)--(0,0);
				\draw[thin] (1,1)--(2,1)--(1,2)--(2,2)--(1,1);
				\draw[thin] (0,3)--(1,2);
				\draw[thin] (3,3)--(2,2);

				\node at (1.2,.8) {$v_1$};
				\node at (2.2,.8) {$v_2$};
				\node at (2.2,1.8) {$v_3$};
				\node at (.8,1.8) {$v_4$};
				\node at (.3,3) {$v_8$};
				\node at (2.7,3) {$v_5$};
				\node at (.1,1.2) {$v_6$};
				\node at (1.1,-.25) {$v_9$};
				\node at (-.1,-.2) {$v_7$};
				\node at (.6,-1) {$G$};
			\end{tikzpicture}
			\hspace*{-1cm}
			\vfil
			$\bordermatrix{%
				&v_1&v_2&v_3&v_4&v_5&v_6&v_7&v_8&v_9\cr
				v_1& 0&0&0&0& 0&0&0&3&0 \cr
				v_2& 0&0&0&0& 3&0&0&0&3 \cr
				v_3& 0&0&0&0& 0&0&0&0&3 \cr
				v_4& 0&0&0&0& 0&0&0&0&4 \cr
				v_5& 0&3&0&0& 0&0&0&0&4 \cr
				v_6& 0&0&0&0& 0&0&0&4&0 \cr
				v_7& 0&0&0&0& 0&0&0&4&0 \cr
				v_8& 3&0&0&0& 0&4&4&0&5 \cr
				v_9& 0&3&3&4& 4&0&0&5&0 \cr		
			}.$
		\end{center}
	\end{multicols}
	
	\noindent The eigenvalues of $E(G)$ are $0$ with multiplicity $3$ and the 
	remaining simple eigenvalues are $-9.4967,~ -4.3784,~ -2.9329,~ 1.4150,~ 
	5.2920,~  10.1010$, which are computed using SAGEMATH.  
	Therefore, $\In(E(G))=(3,3,3)$ and $E(G)$ is not symmetric about the origin. 
\end{example}
\subsection{Irreducibility of  eccentricity matrices 
	of graphs in $\mathscr{B}$}
The problem of characterizing 
graphs 
whose eccentricity matrices are irreducible remains open. So far, only a few classes of 
graphs 
have been identified whose eccentricity matrices are irreducible, see
\cite{LiWB2022,MGRA2021,PaSP2021,WLBR2018,WLLB2020}. In this subsection, we 
prove 
that the 
eccentricity matrices of  graphs in $\mathscr{B}$ are irreducible.

For our purpose, let us recall the following lemma, which gives an equivalent 
condition for the 
irreducibility of a non-negative matrix.
\begin{lemma}(see \cite{MGKA2020})\label{irr_equ_con}
	Let $M=(m_{ij})$  be an $n \times n$ nonnegative symmetric matrix and $G(M)$ be 
	the graph on $n$ vertices such that there is an edge between the vertices $i$ 
	and $j$ in $G(M)$ if and only if $m_{ij}\neq 0$. Then 
	$M$ is irreducible if and only if $G(M)$ is connected.
\end{lemma}

\begin{theorem}\label{irre_thm}
	Let $G\in \mathscr{B}$ with $\diam(G)\geq 4$. Then the 
	eccentricity matrix $E(G)$ is irreducible.
\end{theorem}
\vspace{-.5cm}
\begin{proof}
	Let	$G\in \mathscr{B}$.
	Using Lemma \ref{irr_equ_con} and the matrices given in  
	(\ref{odd_case1}),(\ref{odd_case2}),(\ref{even_case2}) and (\ref{even_case1}),  
	it is direct to see  that $E(G)$ is irreducible.
\end{proof}

\section*{Conclusion}
The inertias of the eccentricity matrices of a subclass of bi-block 
graphs $\mathscr{B}$ (which contains trees) are derived by 
associating a tree $T_G$ for each $G\in \mathscr{B}$. 
A characterization for the spectrum of the eccentricity matrix $E(G)$ is symmetric about the origin being given in terms of the diameter of $G$.
Also, it is proved that the eccentricity matrix  of $G\in 
\mathscr{B}$ with $\diam(G)\geq 4$ is irreducible.

\bibliographystyle{plain}
\bibliography{Ref}

\end{document}